\documentclass[a4paper, 10pt, twoside]{amsart}

\usepackage[utf8x]{inputenc}

\usepackage{amsmath}
\usepackage{amssymb}
\usepackage{amsthm}
\usepackage{amscd}
\usepackage[british]{babel}

\usepackage{enumerate}
\usepackage{lefsymp}

\title[Lefschetz numbers of symplectic involutions]{Lefschetz numbers of symplectic involutions on arithmetic groups}
\author[S. Kionke]{Steffen Kionke}
\date{\today}

\address{Max-Planck Institute for Mathematics, Vivatsgasse 7, 53111 Bonn, Germany}
\email{skionke@mpim-bonn.mpg.de}
\thanks{The author was supported by FWF Austrian Science Fund, grant P 21090-N13.}
\begin{document}
\selectlanguage{british}
 %%%%%%%% abstract %%%%%%%%%
  \begin{abstract}
	The reduced norm-one group $G$ of a central simple algebra is an inner form of the special linear group,
	and an involution on the algebra induces an automorphism of $G$.
	We study the action of such automorphisms in the cohomology of arithmetic subgroups of $G$.
	The main result is a precise formula for Lefschetz numbers of automorphisms induced by involutions of symplectic type.
	Our approach is based on a careful study of the smoothness properties of group schemes associated with orders in central simple algebras.
	Along the way we also derive an adelic reformulation of Harder's Gau\ss-Bonnet Theorem.
  \end{abstract}
 %%%%%%%%%%%%%%%%%%%%%%%%%%%
\maketitle

\section{Introduction}\label{sec:intro}
Let $G$ be a semisimple linear algebraic group defined over the field $\bbQ$ of rational numbers.
Given a torsion-free arithmetic subgroup $\Gamma \subset G(\bbQ)$ it is in general a very difficult task to compute the (cohomological) Betti numbers of $\Gamma$.
However Harder's Gau\ss-Bonnet Theorem \cite{Harder1971} makes it possible to determine the Euler characteristic of arithmetic groups.
If the Euler characteristic is non-zero, one can extract information on the Betti numbers. Moreover,
whether or not the Euler characteristic vanishes only depends on the structure of the associated real Lie group $G(\bbR)$ (see Rem.~\ref{rem:NoComplexPlaces}).
If the Euler characteristic vanishes, Lefschetz numbers of automorphisms of finite order of $G$ are a suitable substitute to gain insight into the cohomology of $\Gamma$.
The idea to study Lefschetz numbers in the cohomology of arithmetic groups goes back to Harder \cite{Harder1975}. A general method was developed by J.~Rohlfs, first for Galois automorphisms
 \cite{Rohlfs1978} and later in a general adelic setting~\cite{Rohlfs1990}. Lefschetz numbers were also studied by Lee-Schwermer \cite{LeeSchwermer1983} and Lai \cite{Lai1991}.
However, only very few groups have been considered in detail, most frequently Lefschetz numbers on Bianchi groups were studied (see Krämer \cite{Kraemer1985},
Rohlfs \cite{Rohlfs1985},  Seng\"un-T\"urkelli \cite{SengunTurkelli2012}, and Kionke-Schwermer \cite{KSchwermer2012}).
In this article we describe a method (based on Rohlfs' approach) to compute Lefschetz numbers of specific automorphisms on arithmetic subgroups of inner forms of the special linear group.

More precisely, let $F$ be an algebraic number field and let $A$ be a central simple $F$-algebra.  The reduced norm $\nrd_{A/F}$ is a polynomial function on $A$
and the associated reduced norm-one group $G = \SL_A$ is a linear algebraic group defined over $F$.  Indeed, the algebraic group $G$ is an
inner form of the special linear group.
If $A$ has an involution $\sigma$ of symplectic type (see Def.~\ref{def:symplecticType}), then
the composition of $\sigma$ with the group inversion yields an automorphism $\sigma^*$ of $G$.
We study the Lefschetz numbers of such automorphisms induced by involutions of symplectic type.

\subsection{The main result}
Let $F$ be an algebraic number field and let $\calO$ denote its ring of integers.
Let $A$ be a central simple $F$-algebra. For our purposes we may assume that $A = M_n(D)$ for some quaternion $F$-algebra $D$ (see \ref{sec:redQuat}).

Let $\Lambda_D \subseteq D$ be a maximal $\calO$-order in $D$,
then $\Lambda := M_n(\Lambda_D)$ is a maximal $\calO$-order in~$A$. 
For a non-trivial ideal $\LA \subseteq \calO$ we study the cohomology of the principal congruence subgroups
\begin{equation*}
   \Gamma(\LA) := \{\:g \in M_n(\Lambda_D)\:|\:\nrd_A(g) = 1 \:\text{ and }\: g \equiv 1 \bmod \LA\:\}
\end{equation*}
of $G$.  In fact, for $n\geq 2$ the groups $\Gamma(\LA)$ have vanishing Euler characteristic.

The quaternion algebra $D$ is equipped with a unique involution of symplectic type $\tau_c: D \to D$, called \emph{conjugation},
which induces an involution of symplectic type $\tau: A \to A$ by $\tau(x) := \tau_c(x)^T$, i.e.~apply $\tau_c$ to every entry of the matrix and then transpose the matrix. 
We will call $\tau$ the \emph{standard involution of symplectic type} on $M_n(D)$.
Composition of $\tau$ with the group inversion yields an automorphism $\tau^*$ of order two on $G$.
Note that the congruence groups $\Gamma(\LA)$ are $\tau^*$-stable. Fix a rational representation $\rho: G\times_F \alg{F} \to \GL(W)$ of $G$ (defined over the algebraic closure of $F$)
 on a finite dimensional vector space. If $W$ is equipped with a compatible $\tau^*$-action (see Def.~\ref{def:compatibleAction}), then
we can define the Lefschetz number $\mcal{L}(\tau^*,\Gamma(\LA),W)$ of $\tau^*$ in the cohomology $\com{H}(\Gamma(\LA),W)$.
\begin{maintheorem}
  Assume that $\Gamma(\LA)$ is torsion-free.
   If $D$ is totally definite, we assume further that $n \geq 2$.
   The Lefschetz number $\mcal{L}(\tau^*,\Gamma(\LA),W)$ is 
   zero if $F$ is not totally real. 
   
   If $F$ is totally real, the Lefschetz number
   is given by the following formula
   \begin{equation*}
       \mcal{L}(\tau^*,\Gamma(\LA),W) \: = \:
         2^{-r} \N(\LA)^{n(2n+1)} \Drd{D}^{n(n+1)/2} \Tr(\tau^*|W) \prod_{j=1}^n M(j,\LA, D).
   \end{equation*}
    Here $\Drd{D}$ denotes the signed reduced discriminant of $D$ (see Def.~\ref{def:redDiscr}), $r$ denotes the number of real places of $F$
    ramified in $D$, and 
    \begin{equation*}
      M(j,\LA,D) := \zeta_F(1-2j) \prod_{\LP | \LA}\bigl(1-\frac{1}{\N(\LP)^{2j}}\bigr)
      \prod_{\substack{\LP \in \Ram_f(D) \\ \LP \nmid \LA}} \bigl(1+(\frac{-1}{\N(\LP)})^{j}\bigr)
   \end{equation*}  
      where $\Ram_f(D)$ denotes the set of finite places of $F$ where $D$ ramifies and $\zeta_F$ denotes the Dedekind zeta-function of $F$.
      If $F$ is totally real, then the Lefschetz number is zero if and only if $\Tr(\tau^*|W) = 0$. 
\end{maintheorem}

\subsection{Applications}
We briefly give three applications of the above formula where we always assume $F$ to be a totally real number field.

\subsubsection{Growth of the total Betti number}
The analysis of the asymptotic behaviour of Betti numbers of arithmetic groups is an important topic. Recent results of Calegari-Emerton
provide strong asymptotic upper bounds (cf.~\cite{CalegariEmerton2009}).
We can use the main theorem to obtain an asymptotic lower bound result.

Let $G$ be the reduced norm-one group associated with the central simple $F$-algebra $M_n(D)$. For a torsion-free arithmetic subgroup $\Gamma \subseteq G(F)$
we define the \emph{total Betti number} $B(\Gamma) := \sum_{i = 0}^\infty \dim H^i(\Gamma,\bbC)$. Note that this is a finite sum since torsion-free arithmetic
groups are of type (FL) (see \cite[Thm.~11.4.4]{BorelSerre1973}).
\begin{corollary}\label{cor:GrowthBettiNumber}
 Let $\Gamma_0\subset G(F)$ be an arith\-metic subgroup.
 For any ideal $\LA \subset \calO$ we define
  $\Gamma_0(\LA) := \Gamma_0 \cap \Gamma(\LA)$.
 There is a positive real number $\kappa>0$, depending on $F$, $D$, $\Gamma_0$, and $n$, 
 so that
  \begin{equation*}
     B(\Gamma_0(\LA)) \geq \kappa [\Gamma_0:\Gamma_0(\LA)]^{\frac{n(2n+1)}{4n^2-1}}
  \end{equation*}
  holds for every ideal $\LA$ such that $\Gamma(\LA)$ is torsion-free.
\end{corollary}
A proof of this corollary will be given in Section \ref{sec:growthOfBettinumber}.

\subsubsection{Rationality of zeta values} Note that the Lefschetz number is an integer since $\tau^*$ is of order two.
We obtain a new proof of a classical theorem of Siegel \cite{Siegel1969} and Klingen \cite{Klingen1962}.
\begin{corollary}
 If $F$ is a totally real number field, then $\zeta_F(1-2m)$ is a non-zero rational number for all integers $m \geq 1$.
\end{corollary}
\begin{proof}
 Apply the main theorem with $D = M_2(F)$, $\Lambda_D = M_2(\calO)$ and choose $W$ to be the trivial one-dimensional representation.
  We see that for every $n \geq 1$ and all sufficiently small ideals $\LA \subseteq \calO$, the number
  \begin{equation*}
       \N(\LA)^{n(2n+1)} \prod_{j=1}^n\Bigl( \zeta_F(1-2j) \prod_{\LP | \LA}\bigl(1-\frac{1}{\N(\LP)^{2j}}\bigr) \Bigr)
  \end{equation*}
  is a non-zero integer. The claim follows by induction on $m$.
\end{proof}

\subsubsection{Cohomology of cocompact fuchsian groups}
 Let $D$ be a division quaternion algebra over $F$
such that $D$ is split at precisely one real place $v_0$ of $F$.
So $r = [F:\bbQ]-1$ is the number of real places ramified in $D$.

Let $\Lambda = \Lambda_D$ be a maximal $\calO$-order in $D$. We consider the reduced norm-one group $G = \SL_D$ defined over $F$.
The associated real Lie group is
\begin{equation*}
   G_\infty \cong \SL_2(\bbR) \times \SL_1(\bbH)^r.
\end{equation*}
Note that the group $\SL_1(\bbH)$ is compact and so the projection $p_1 : G_\infty \to \SL_2(\bbR)$ onto the first factor 
is a proper and open homomorphism of Lie groups. In particular, every discrete torsion-free subgroup $\Gamma \subseteq G_\infty$
maps via $p_1$ isomorphically to a discrete subgroup in $\SL_2(\bbR)$. 

Let $\LA \subseteq \calO$ be a proper ideal such that $\Gamma(\LA)$ is 
torsion-free.
We will interpret $\Gamma(\LA)$ as a subgroup of $\SL_2(\bbR)$. Note that, since we assumed $D$ to be a division algebra, the 
group $\Gamma(\LA)$ is a cocompact Fuchsian group (see Thm.~5.4.1 in \cite{Katok1992}). 

Let $\LH = \SL_2(\bbR)/\SO(2)$ be the Poincar\'e upper half-plane.
\begin{corollary}\label{cor:genusFuchsian}
  The compact Riemann surface 
  $\LH/\Gamma(\LA)$ has genus
  \begin{equation*}
       g = 1+ 2^{-[F:\bbQ]}\N(\LA)^3 \left|\Drd{D} \zeta_F(-1)\right| \prod_{\LP | \LA} \bigl(1-\N(\LP)^{-2}\bigr) 
       \prod_{\substack{\LP \in \Ram_f(D) \\ \LP \nmid \LA}} \bigl(1-\N(\LP)^{-1}\bigr)
  \end{equation*}
   This implies an explicit formula for the first Betti number $b_1(\Gamma(\LA))$ since
   \begin{equation*}
       b_1(\Gamma(\LA))= \dim H^1(\Gamma(\LA),\bbC) = 2g. 
   \end{equation*}
\end{corollary}
\begin{proof}
   Consider the main theorem for $n=1$. Note that for $n=1$ the automorphism $\tau^*$ is actually the identity.
   This means that, using the main theorem with the trivial representation, we obtain
   \begin{equation*}
       \mcal{L}(\tau^*, \Gamma(\LA), \bbC) = \chi(\Gamma(\LA)) = \chi(\LH / \Gamma(\LA)).
   \end{equation*}
   Note that the sign of the Lefschetz number is $-1$. 
   Since $\chi(\LH/\Gamma(\LA)) = 2 - 2g$, the claim follows immediately.
\end{proof}
In fact Corollary \ref{cor:genusFuchsian} yields a precise formula for the dimension of the space of holomorphic weight $k$
modular forms for the group $\Gamma(\LA)$ (use Shimura's Theorems 2.24 and 2.25 in \cite{Shimura1971}). 

\subsection{Reduction to quaternion algebras}\label{sec:redQuat}
Let $A$ be a central simple $F$-algebra.
If $A$ has an involution $\sigma$ of symplectic type (see Def.~\ref{def:symplecticType}),
then $A$ is isomorphic to the opposed $F$-algebra $A^{op}$. This means that
the class of $A$ has order two in the Brauer group of $F$. 
Since the dimension of $A$ is even, it follows from (32.19) in~\cite{Reiner2003} that $A$ is isomorphic to a matrix algebra $M_n(D)$ over
a quaternion algebra~$D$. Therefore we always assume $A = M_n(D)$.

Let $\tau$ be the standard involution of symplectic type on $M_n(D)$. 
Note further that 
in this case ${\sigma = \inn(g) \circ \tau}$ for an element 
$g \in \GL_n(D)$ with $\tau(g) = g$.
Due to this observation it is only a minor restriction if
we focus on the standard symplectic involution $\tau$.

\subsection{Structure of this article}
In Section \ref{sec:smoothGroupSchemes} give a short general treatment of smooth group schemes over Dedekind rings which are associated with 
orders in central simple algebras. In particular we treat integral models of inner forms of the special linear group.
Further, we consider the fixed points groups attached to involutions. An important tool in the proof of the main theorem will be
the pfaffian as a map in non-abelian Galois cohomology (cf.~Section~\ref{sec:symplecticAndPfaffian}).
In Section~\ref{sec:HardersGB} we give an adelic reformulation of Harder's Gau\ss-Bonnet Theorem which hinges on the notion of smooth group scheme.
The calculation of the Lefschetz number is based on Rohlfs' method which we summarise in Section~\ref{sec:RohlfsMethod}.
Finally the proof of the main theorem is contained in Section~\ref{sec:ProofMain}. It consists of two major steps. The first is
the analysis of various non-abelian Galois cohomology sets which occur in Rohlfs' decomposition. The second step is the calculation of the Euler characteristics of the fixed point groups using 
Harder's Gau\ss-Bonnet Theorem.

\subsection*{Notation}
Apart from Section \ref{sec:smoothGroupSchemes}, where we work in a more general setting, we use the following notation:
$F$ is an algebraic number field and $\calO$ denotes its ring of integers.
Let $V$ denote the set of places of $F$. We have $V = V_\infty \cup V_f$ where
$V_\infty$ (resp. $V_f$) denotes the set of archimedean (resp. finite) places of $F$.
Let $v \in V$ be a place of~$F$, then we denote the completion of $F$ at $v$ by $F_v$.
The valuation ring of $F_v$ is denoted by $\calO_v$ and the prime ideal in $\calO_v$ is denoted by $\LP_v$. 
For a non-zero ideal $\LA \subseteq \calO$ the ideal norm is defined by $\N(\LA) := |\calO/\LA|$.
As usual $\bbA$ denotes the ring of adeles of $F$ and $\bbA_f$ is the ring of finite adeles.

\section{Group schemes associated with orders in central simple algebras}\label{sec:smoothGroupSchemes}

In this section we will investigate the smoothness properties of group schemes attached to orders in central simple algebras.
Throughout $R$ denotes a Dedekind ring and $k$ denotes its field of fractions. For simplicity we assume $\chr(k) = 0$.
In our applications $R$ is usually the ring of integers of an algebraic number field or a complete discrete valuation ring. 

The term \emph{scheme} always refers to an affine scheme of finite type, the same holds for group schemes.
Recall that a scheme $\schX$ defined over $R$ is \emph{smooth} if
for every commutative $R$-algebra $C$ and every nilpotent ideal $I \subseteq C$ the 
induced map $\schX(C) \to \schX(C/I)$ is surjective.
Suppose $R$ is a complete discrete valuation ring and let $\LP$ denote its prime ideal.
We will frequently use the following property: If $\schX$ is a smooth $R$-scheme, then the 
induced map $\schX(R) \to \schX(R/\LP^e)$ is surjective for every integer $e \geq 1$ (cf.~Cor.~19.3.11, EGA IV, \cite{EGA4-1}).
If $G$ is a group scheme, then we denote the Lie algebra of $G$ by $\Lie(G)$.

\subsection{The general linear group over an order}\label{sec:unitgroups}
Let $A$ be a central simple $k$-algebra and let $\Lambda$ be an $R$-order in $A$.
Since $\Lambda$ is a finitely generated torsion-free $R$-module, it is a finitely generated \emph{projective} $R$-module (cf.~(4.13) in \cite{Reiner2003}).
The functor $\Lambda_a$ from the category of commutative $R$-algebras to the category of rings
defined by $C \mapsto \Lambda \otimes_R C$ is represented by the symmetric algebra $S_R(\Lambda^*)$ where
$\Lambda^* = \Hom_R(\Lambda, R)$. In fact it defines a smooth $R$-scheme (cf.~19.3.2 in EGA~IV~\cite{EGA4-1}).

Recall that, since $\Lambda$ is finitely generated and projective, one can attach to every $R$-linear endomorphism $\varphi$ of $\Lambda$ its determinant $\det(\varphi) \in R$.
More precisely, here the determinant of $\varphi$ is just the determinant of the $k$-linear extension $\varphi\otimes\id_k: A \to A$.
As usual one defines the \emph{norm} of an element $x \in \Lambda$ to be the determinant of the left-multiplication with $x$.
One can check that the norm defines a morphism of schemes over $R$
 \begin{equation*}
     \N_{\Lambda/R}: \Lambda_a \to \bbA^1/R
 \end{equation*}
to the affine line $\bbA^1$ defined over $R$. This can be seen, for instance, by observing that the norm is a natural transformation of functors.
Let $C$ be a commutative $R$-algebra. An element $x \in \Lambda \otimes_R C$ is a unit if and only if $\N_{\Lambda/R}(x) \in C^\times$.
It follows from the next lemma that the associated unit group functor $\GL_\Lambda: C \mapsto (\Lambda \otimes_R C)^\times$
is a smooth group scheme over $R$.
\begin{lemma}\label{lem:schemeUnits}
  Let $\bbA^1$ denote the affine line over $R$.
  Let $\schX$ be an affine scheme over $R$ with a morphism $f: \schX \to \bbA^1$. The subfunctor $\schY$
  (from the category of commutative $R$-algebras to the category of sets)
  defined by 
       \begin{equation*}
       C \mapsto \{\:y \in \schX(C) \:|\: f(y) \in C^\times\}
       \end{equation*}
  is an affine scheme and
  the natural transformation $\schY \to \schX$ is a morphism of schemes.
  If $\schX$ is smooth, then $\schY$ has the same property. 
\end{lemma}
\begin{proof}
  Let $R[\schX]$ be the coordinate ring of $\schX$ and let $P \in R[\schX]$ be the polynomial defining $f$.
  Note that $\schY$ is canonically isomorphic to the functor 
  \begin{equation*}
    C \mapsto \{\:(y,z) \in \schX(C)\times C\:|\: f(y)z = 1\:\}.
  \end{equation*}
  Using this it is easily checked that the $R$-algebra $S := R[\schX]\otimes_R R[T]/(P\otimes T - 1)$ represents $\schY$.
  Clearly, $S$ is of finite type since $R[\schX]$ is of finite type.

  It remains to show that $\schY$ is smooth, whenever $\schX$ is smooth.
  Assume $\schX$ to be smooth and take a commutative $R$-algebra $C$ with an ideal $J$ such that $J^2=0$.
  By assumption $\schX(C) \to \schX(C/J)$ is surjective, so given $y \in \schY(C/J)$ we find $x \in \schX(C)$ projecting
  onto $y$. By assumption $f(x) +J$ is a unit in $C/J$. In particular, we find $z \in C$ with $f(x)z  \in 1+ J$. However, $1+J$ consists
  entirely of units and thus $f(x) \in C^\times$. We deduce that $\schY$ is smooth.
\end{proof}

To stress this once more: in this article $\GL_\Lambda$ is always a functor and not a group.
If we take $\Lambda = R$ then we call $\GL_\Lambda$ the \emph{multiplicative group} (or multiplicative group scheme) defined over $R$, and we denote it by $\Gm$.
Note that the norm defines a homomorphism of $R$-group schemes
 \begin{equation*}
     \N_{\Lambda/R}: \GL_\Lambda \to \Gm.
 \end{equation*}
We also point out that the Lie algebra of $\GL_\Lambda$ can be (and will be) identified with $\Lambda_a$ in a natural way.

\subsection{The special linear group over an order}

\subsubsection{Reduced norm and trace} \label{par:redNormTrace}
Let $A$ be a central simple $k$-algebra, we consider the reduced norm and trace (for definitions see section 9 in \cite{Reiner2003} or IX, §2 in \cite{WeilBNT}).
It was observed by Weil that the reduced norm and trace are polynomial functions. We reformulate this in schematic language:
There is a unique element $\nrd_{A/k}$ in the symmetric algebra $S_k(A^*)$ (here $A^* = \Hom_k(A,k)$) such that 
for every splitting field $\ell$ of $A$ and every splitting $\varphi: A\otimes_k \ell \isomorph M_n(\ell)$
the induced map $S(\varphi^*): S_\ell(M_n(\ell)^*) \to S_k(A^*)\otimes_k \ell$ maps the determinant to $\nrd_{A/k}\otimes 1$. 
Similarly there is the reduced trace $\trd_{A/k} \in A^*$ with an analogous property.

Let $\Lambda \subseteq A$ be an $R$-order. We show that the reduced norm and trace are defined over $R$ in the appropriate sense.
For the reduced trace this is easy: Elements in $\Lambda$ are integral over $R$, hence the reduced trace takes values in $R$ on the order $\Lambda$ and 
defines an $R$-linear map $\Lambda \to R$.
In particular we obtain a morphism of schemes over $R$
\begin{equation*}
   \trd_{\Lambda/R}: \Lambda_a \to \bbA^1 / R.
\end{equation*}
Consider the reduced norm. From (9.7) in \cite{Reiner2003} one can deduce that $\nrd_{A/k}^n$ and $\N_{\Lambda/R}$ agree as elements in the coordinate ring $S_k(A^*)$.
However, the coordinate ring $S_R(\Lambda^*)$ of $\Lambda_a$ is integrally closed in $S_k(A^*)$ and we conclude that the reduced norm is defined over $R$.
This means that there is a polynomial $\nrd_{\Lambda/R} \in S_R(\Lambda^*)$ which defines the reduced norm as
a morphism of $R$-schemes
\begin{equation*}
     \nrd_{\Lambda/R}: \Lambda_a \to \bbA^1.
\end{equation*}
We can also restrict the reduced norm to the unit group and obtain a homomorphism of group schemes
\begin{equation*}
    \nrd_{\Lambda/R}: \GL_\Lambda \to \Gm / R. 
\end{equation*}
\begin{definition}\label{def:SL}
 The \emph{special linear group} $\SL_\Lambda$ over the order $\Lambda$ is the group scheme over $R$ defined by the kernel of the reduced norm,
this is
\begin{equation*}
   \SL_\Lambda = \ker(\nrd_{\Lambda/R}: \GL_\Lambda \to \Gm).
\end{equation*}
\end{definition}

\subsubsection{Smoothness of the special linear group}
Whereas the general linear group is always smooth, independent of the chosen order, the smoothness of the special linear group depends on the underlying order.
Recall the following useful result.
\begin{proposition}[Smoothness of kernels] \label{prop:SmoothKernels}
 Let $f: G \to H$ be a morphism between two smooth group schemes over $R$.
 If the derivative $\diff(f): \Lie(G)(R) \to \Lie(H)(R)$ is surjective, then the group scheme $K := \ker(f)$ is 
 smooth over $R$. 
\end{proposition}
\begin{proof}
    This follows from the theorem of infinitesimal points (see \cite[p.~208]{DemazureGabriel}) and some easy diagram chasing.
\end{proof}

As a matter of fact the derivative of the reduced norm $\diff(\nrd_{\Lambda/R}): \Lambda_a \to \bbA^1$ is the reduced trace.
Having this in mind we make the following definition.
\begin{definition}\label{def:smooth}
  An $R$-order $\Lambda$ in a central simple $k$-algebra is called \emph{smooth} if the reduced trace
  $\trd_{\Lambda/R}\!: \Lambda \to R$ is surjective.
\end{definition}
 Note that smoothness of orders is a local property.
\begin{corollary}\label{cor:smoothSL}
  If the order $\Lambda$ is smooth then the scheme $\SL_\Lambda$ is smooth.
\end{corollary}
\begin{proof}
  This follows immediately from Proposition \ref{prop:SmoothKernels} using the fact that the derivative of the reduced norm is the reduced trace.
\end{proof}
 In fact, also the converse statement holds under the assumption $\chr(R)=0$. However, we shall not need this result.
 The next proposition shows that smooth orders exist.
 \begin{proposition}\label{prop:maxImpliesSmooth}
   Assume that $R/\LP$ is finite for every prime ideal $\LP$. Then every maximal $R$-order in a central simple $k$-algebra is smooth.
 \end{proposition}
 \begin{proof}
   Let $A$ be a central simple $k$-algebra and let $\Lambda\subset A$ be a maximal $R$-order.
   Since $\Lambda$ is maximal in $A$ if and only if all $\LP$-adic completions are maximal orders 
   (see (11.6) in \cite{Reiner2003}), and since smoothness of $\Lambda$ is a local property, we may assume that $R$ is a complete
   discrete valuation ring.
   Recall that $A$ is isomorphic to a matrix algebra $M_r(D)$ over a central division algebra $D$. Moreover, $D$ has a
   unique maximal $R$-order $\Delta \subseteq D$ and $\Lambda$ is (up to conjugation) the maximal order $M_r(\Delta)$ in $A$
   (see (17.3) in \cite{Reiner2003}). It is known that the reduced trace of a matrix $x =(x_{ij})_{i,j=1}^r \in M_r(D)$ 
   is given by 
   \begin{equation*}
           \trd_{A/k}(x) = \sum_{i=1}^r \trd_{D/k}(x_{ii})
   \end{equation*}
    (cf.~Cor.~2, IX.§2 in \cite{WeilBNT}). Hence we may assume that $A = D$ is a division algebra and
    $\Lambda = \Delta$ is the unique maximal order. Let $\dim_k D = n^2$ and let $\ell/k$ be the unique unramified 
    extension of $k$ of degree $[\ell:k] = n$. The field $\ell$ embeds into $D$ as a maximal subfield and 
    the reduced trace $\trd_{D/k}$ on the elements of $\ell$ agrees with the field trace $\Tr_{\ell/k}$ (cf.~proof of (14.9)
 in \cite{Reiner2003}).
    Let $o_\ell$ denote the valuation ring of $\ell$. The image of $o_\ell$ under the embedding $\ell \to D$ 
    lies in the maximal order $\Delta$. Finally the surjectivity of $\trd_{D/k}: \Delta \to R$ follows
    from the well-known surjectivity of the field trace $\Tr_{\ell/k}: o_\ell \to R$.
\end{proof}
 
 \subsection{Involutions and fixed point groups}\label{sec:FixedPointGroups}
  Let $A$ be a central simple $k$-algebra. 
  An \emph{involution} $\tau$ on $A$ is 
  an additive mapping $\tau: A \to A$ of order two such that $\tau(xy) = \tau(y) \tau(x)$ for all $x,y \in A$.
  We say that $\tau$ is of the \emph{first kind} if $\tau$ is $k$-linear. Otherwise, we say that $\tau$ is of the \emph{second kind}. In this article all involutions are of the first kind unless
  the contrary is explicitly stated. We will mostly focus on involutions of symplectic type.
  \begin{definition}\label{def:symplecticType}
  We say that an involution $\tau$ on $A$ is of \emph{symplectic type}, if
  there is a splitting field $\ell$ of the algebra $A$, a splitting 
  \begin{equation*}
        \varphi: A \otimes_k \ell \isomorph M_{2n}(\ell)
  \end{equation*}
  and a skew symmetric matrix $a \in M_{2n}(\ell)$ satisfying $\varphi(\tau(x)) = a \varphi(x)^T a^{-1}$ for all
  elements $x \in A\otimes_k \ell$. If this is the case, then every splitting (over any splitting field) has this property.
  \end{definition}
  
 Let $\tau: A \to A$ be an involution of the first kind. Let $\Lambda$ be an $R$-order in $A$ and assume that $\Lambda$ is $\tau$-stable.
 Since $\tau: \Lambda \to \Lambda$ is $R$-linear, we obtain a morphism of $R$-schemes
 \begin{equation*}
    \tau: \Lambda_a \to \Lambda_a.
 \end{equation*}
 We restrict $\tau$ to the unit group $\GL_\Lambda$ and compose it with the group inversion to obtain a homomorphism of group schemes
 \begin{equation*}
    \tau^*: \GL_\Lambda \to \GL_\Lambda.
 \end{equation*}
 We define $G(\Lambda,\tau)$ to be the group of fixed points of $\tau^*$, this is, for every commutative $R$-algebra $C$ we obtain
 \begin{equation*}
     G(\Lambda,\tau)(C) = \{\:x \in (\Lambda\otimes_R C)^\times\:|\:\tau(x)x=1\:\}.
 \end{equation*}
  We analyse the smoothness properties of group schemes constructed in this way. Define $\Sym(\Lambda,\tau) = \{\:x \in \Lambda\:|\:\tau(x) = x\:\}$ and note that this $R$-submodule
  of $\Lambda$ is even a direct summand of.
  \begin{lemma}\label{lem:symmetricElem}
     For every commutative $R$-algebra $C$, every $y \in \Lambda \otimes_R C$ and every $x \in \Sym(\Lambda,\tau) \otimes_R C$
     we have $\tau(y)xy \in \Sym(\Lambda,\tau)\otimes_R C$.
  \end{lemma}
  \begin{proof}
  We can write $y = \sum_i u_i \otimes c_i$ for certain $u_i \in \Lambda$ and $c_i \in C$.
  The claim is linear in $x$, hence we may assume $x = e \otimes c$ with $e \in \Sym(\Lambda,\tau)$ and $c \in C$.
  We calculate
  \begin{align*}
     \tau(y) x y &= \sum_{i,j} \tau(u_i)e u_j \otimes c c_i c_j \\
                 &= \sum_{i} \tau(u_i)e u_i \otimes c c_i^2 + \sum_{i < j} (\tau(u_i)e u_j +  \tau(u_j)e u_i)\otimes c c_i c_j,\\
  \end{align*}
  and we see that $\tau(y)xy$ is in $\Sym(\Lambda,\tau) \otimes_R C$ since $\tau(u_i)eu_i$ and $\tau(u_i)e u_j +  \tau(u_j)e u_i$ are elements of $\Sym(\Lambda,\tau)$.
  \end{proof}

  \begin{definition}\label{def:tausmooth}
     The order $\Lambda$ is called \emph{$\tau$-smooth} if the map $s: \Lambda \to \Sym(\Lambda, \tau)$ defined by $x \mapsto x + \tau(x)$ is surjective.
  \end{definition}
  Clearly $\tau$-smoothness is a local property.
  
  \begin{proposition}\label{prop:smoothFPGroups}
     If an $R$-order $\Lambda$ is $\tau$-smooth, then the scheme $G(\Lambda,\tau)$ is smooth.
  \end{proposition}
   \begin{proof}
   We set $G := G(\Lambda,\tau)$. Let $C$ be a commutative $R$-algebra with an ideal $I \subseteq C$ such that $I^2=0$.
   We have to show that the canonical map $G(C) \to G(C/I)$ is surjective.
   Take $\bar{y} \in G(C/I)$. Since the unit group scheme $\GL_\Lambda$ is smooth (see \ref{sec:unitgroups}),
    we find $y \in \GL_\Lambda(C) = (\Lambda\otimes C)^\times$ which maps to $\bar{y}$ modulo $I$.
   Since $\bar{y}$ is in the fixed point group of $\tau^*$, this implies that
   \begin{equation*}
        \tau(y)y = 1 + \rho
   \end{equation*}
   with some $\rho \in \Lambda \otimes I$. 

  We consider $E:= \Sym(\Lambda, \tau)$ and we obtain $\tau(y)y \in E \otimes_R C$ by Lemma \ref{lem:symmetricElem}.
  Consequently, there is $u \in \Lambda\otimes_R C$ such that $\tau(u)+u = y$. 
  Moreover, we have $1 \in E$, thus there is some $v \in \Lambda\otimes_R C$ with $\tau(v) + v = 1$.
  We deduce that $\rho = \tau(u-v) + (u-v)$ is an element in $E \otimes_R C$, and thus
  \begin{equation*}
      \rho \in (E \otimes_R C) \cap (\Lambda \otimes_R I) = E \otimes_R I
  \end{equation*}
  As last step we use once again that $\Lambda$ is $\tau$-smooth and deduce that
  there is some $w \in \Lambda \otimes I$ with $\rho = \tau(w) + w$.
  We put $y' := y(1-w)$, which is congruent $\bar{y}$ modulo $I$ and satisfies
  \begin{equation*}
     \tau(y')y' = (1-\tau(w)) \tau(y)y (1-w) = (1-\tau(w))(1 +\rho)(1-w) =1 + \rho - \tau(w)- w = 1.
  \end{equation*}
  Therefore $y' \in G(C)$ and $y'$ maps to $\bar{y} \in G(C/I)$ under the canonical map.
\end{proof}
  It is possible to prove also the converse statement, however, this will not be needed in the sequel.

\subsection{Involutions of symplectic type and the pfaffian}\label{sec:symplecticAndPfaffian}
Let $A$ be a central simple $k$-algebra with an involution of symplectic type $\tau$. Let $\Lambda$ be a $\tau$-stable $R$-order in $A$.

\subsubsection{The pfaffian}\label{par:pfaffian}
Set $E:= \Sym(\Lambda,\tau)$ in the notation of section \ref{sec:FixedPointGroups}.
The inclusion $\iota: E \to \Lambda$ induces a morphism of $R$-algebras
\begin{equation*}
   S(\iota^*): S_R(\Lambda^*) \to S_R(E^*).
\end{equation*}
Recall that the reduced norm is given by a polynomial function $\nrd_{\Lambda/R} \in S_R(\Lambda^*)$ (see \ref{par:redNormTrace}).
We define $\nrd_{|E} := S(\iota^*)(\nrd_{\Lambda/R}) \in S_R(E^*)$. We will construct a \emph{pfaffian},
 i.e.~a polynomial $\pf_\tau \in S_R(E^*)$
such that $\nrd_{|E} = \pf^2_\tau$. 

Let $L/k$ be any field extension. It follows from (2.9) in \cite{BOInv} that 
for every $x \in E\otimes_R L$ the reduced norm $\nrd_{|E}(x)$ is a square in $L$.
Therefore, we may deduce that there is a polynomial $f \in S_R(E^*)$
such that 
\begin{equation*}
 f^2= \nrd_{|E}.
\end{equation*}
We normalise this polynomial $\pf_\tau := \pm f$ such that $\pf_\tau(1) = 1$ and we call $\pf_\tau$ the \emph{pfaffian} with respect to $\tau$.

 \begin{lemma}\label{lem:pfaffian}
Let $S(\tau^*)$ denote the automorphism of the symmetric $R$-algebra $S_R(\Lambda^*)$ which is induced by $\tau$.
The following assertions hold:
\begin{enumerate}[(i)]
 \item $S(\tau^*)(\nrd_{\Lambda/R}) = \nrd_{\Lambda/R}$, and
 \item for all $y \in \Lambda \otimes_R C$ and all ${x \in \Sym(\Lambda,\tau) \otimes_R C}$, we have 
  \begin{equation*}
    \pf_\tau(\tau(y)xy) = \nrd_{\Lambda/R}(y) \pf_\tau(x), 
  \end{equation*}
   where $C$ is any commutative $R$-algebra.
\end{enumerate}
\end{lemma}
\begin{proof}
  To prove the first claim we may work over fields. However, over fields this is the well-known
  statement (2.2) in \cite{BOInv}. 

  The same proof works for the second statement. Note that $\tau(y)xy$ lies in ${\Sym(\Lambda,\tau) \otimes_R C}$ by Lemma \ref{lem:symmetricElem}.
  Both are polynomial functions on $\Lambda \times \Sym(\Lambda,\tau)$. If they agree over all
  fields then they agree as polynomials. However, over fields this is the result (2.13) in \cite{BOInv}.
\end{proof}

\begin{remark}\label{rem:symplNrdOne}
 Consider the fixed point group scheme $G = G(\Lambda, \tau)$ associated with $\tau$.
Let $x \in G(C)$ for some commutative $R$-algebra $C$. We see from $\tau(x)x = 1$ and Lemma \ref{lem:pfaffian} that
\begin{equation*}
   \nrd_{\Lambda/R}(x) = \pf_\tau(\tau(x)x) = \pf_\tau(1) = 1.
\end{equation*}
Hence the reduced norm restricts to the trivial character on $G(\Lambda,\tau)$.
\end{remark}

\subsubsection{The cohomological pfaffian}
We study non-abelian Galois cohomology of $\tau^*$ with values in the groups $\GL_\Lambda(C)$ and $\SL_\Lambda(C)$. For the definition of non-abelian cohomology
we refer the reader to \cite{Serre1964}, \cite[p.~123-126]{Serre1979} or \cite[Ch.~VII]{BOInv}. We shall in this context often denote $\tau$ and $\tau^*$ by left exponents, i.e.~we write $\up{\tau^*}x$ for $\tau^*(x)$.

Let $C$ be a commutative $R$-algebra and assume that $C$ is flat as $R$-module. 
A \emph{cocycle} $b \in Z^1(\tau^*,\GL_\Lambda(C))$ is an element of $(\Lambda \otimes C)^\times$ which 
satisfies $b \up{\tau^*}b = 1$, or equivalently $b = \up{\tau}b$.
In other words 
\begin{equation*}
   Z^1(\tau^*, \GL_\Lambda(C)) = \Sym(\Lambda\otimes_R C, \tau) \cap \GL_\Lambda(C).
\end{equation*}
The assumption that $C$ is flat yields that
   $\Sym(\Lambda\otimes_R C, \tau) = \Sym(\Lambda, \tau) \otimes_R C$.
 Therefore we can apply the pfaffian associated with $\tau$ 
 to cocycles in $Z^1(\tau^*,\GL_\Lambda(C))$.
 Two cocycles $b$ and $c$ are \emph{cohomologous} if there is $y \in \GL_\Lambda(C)$
 such that $b = \up{\tau}ycy$.
  In this case it follows from Lemma \ref{lem:pfaffian} that $\pf_\tau(b) = \nrd_{\Lambda/R}(y)\pf_\tau(c)$.
  Therefore the pfaffian defines a morphism of pointed sets
  \begin{equation*}
      \pf_\tau: H^1(\tau^*,\GL_\Lambda(C)) \to C^\times/\nrd_{\Lambda/R}(\GL_\Lambda(C)).
  \end{equation*}
   By the same reasoning we obtain a morphism of pointed sets
  \begin{equation*}
      \pf_\tau: H^1(\tau^*,\SL_\Lambda(C)) \to \{\:x \in C^\times\:|\: x^2 = 1\:\}.
  \end{equation*}
  For simplicity we introduce the notation $C^{(2)} := \{\:x \in C^\times\:|\: x^2 = 1\:\}$
  and we define $C^{\times}_\Lambda := \nrd_{\Lambda/R}(\GL_\Lambda(C))$.

  \begin{proposition}[The cohomological diagram for symplectic involutions]\label{prop:diagramSymplectic}
   Let $\tau$ be an involution of symplectic type on $A$ and let $\Lambda$ be a $\tau$-stable $R$-order.
   For every commutative $R$-algebra $C$ which is flat as $R$-module, there is 
   a commutative diagram of pointed sets with exact rows.
   \begin{equation*}
      \begin{CD}
          C^{(2)} \cap C^\times_\Lambda @>{\delta}>> H^1(\tau^*, \SL_\Lambda(C)) @>{j_*}>> H^1(\tau^*, \GL_\Lambda(C)) @>{\nrd}>> C^\times/(C^\times_\Lambda)^2 \\
                 @|                       @V{\pf_\tau}VV                         @V{\pf_\tau}VV                        @|  \\
          C^{(2)} \cap C^\times_\Lambda  @>>>        C^{(2)}                  @>>>        C^\times/C^\times_\Lambda  @>{\cdot^2}>>  C^\times/(C^\times_\Lambda)^2   
       \end{CD}
   \end{equation*}
  The map $\delta$ is injective and the lower row is an exact sequence of groups.
  Here $j_*$ denotes the map induced by the inclusion $j: \SL_\Lambda(C) \to \GL_\Lambda(C)$.
\end{proposition}
\begin{proof}
 The short exact sequence of groups
  \begin{equation*}
      1 \longrightarrow \SL_\Lambda(C) \stackrel{j}{\longrightarrow} \GL_\Lambda(C) \stackrel{\nrd}{\longrightarrow} C^\times_\Lambda \longrightarrow 1
  \end{equation*}
  is even an exact sequence of groups with $\tau^*$-action, where $\tau^*$ acts on $C^\times_\Lambda$ by inversion.
  Consider the initial segment of the associated long exact sequence in the cohomology (see Prop.~I.38 \cite{Serre1964}): 
  \begin{equation*}
      1 \longrightarrow \SL_\Lambda(C)^{\tau^*} \stackrel{j}{\longrightarrow} \GL_\Lambda(C)^{\tau^*} \stackrel{\nrd}{\longrightarrow} C^\times_\Lambda\cap C^{(2)}  \longrightarrow \dots
  \end{equation*}
  It follows from Remark \ref{rem:symplNrdOne} that 
  $\SL_\Lambda(C)^{\tau^*} \stackrel{j}{\longrightarrow} \GL_\Lambda(C)^{\tau^*}$
  is bijective. Thus the long exact sequence takes the form
  \begin{equation*}
      1 \longrightarrow C^\times_\Lambda\cap C^{(2)} \stackrel{\delta}{\longrightarrow} H^1(\tau^*, \SL_\Lambda(C)) \longrightarrow H^1(\tau^*, \GL_\Lambda(C)) \longrightarrow H^1(\tau^*, C^\times_\Lambda).
  \end{equation*}
  It is easy to see that $H^1(\tau^*, C^\times_\Lambda) = C^\times_\Lambda/(C^\times_\Lambda)^2$ which is a subgroup of
  $C^\times/(C^\times_\Lambda)^2$. Hence we simply replace the last term by $C^\times/(C^\times_\Lambda)^2$.
  This yields the upper row of the diagram.
  It is an easy exercise to verify that the lower row is an exact sequence of groups.

  It remains to verify the commutativity of the rectangles.
  The middle one is obviously commutative by definition of the pfaffian in the cohomology.
  For the last rectangle we simply use that $\pf_\tau(g)^2 = \nrd(g)$ for all $g \in Z^1(\tau^*,\GL_\Lambda(C))$ by the construction of the pfaffian.

  Consider the first rectangle. We recall the definition of the connecting morphism~$\delta$:
  Given $c \in C^\times_\Lambda \cap C^{(2)}$, we can find an element $g \in \GL_\Lambda(C)$ such that $\nrd_{\Lambda/R}(g) = c$, then
   $\delta(c)$ is defined to be the class of $g^{-1}\up{\tau^*}g$.
  The pfaffian of $g^{-1}\up{\tau^*}g$ is
   \begin{equation*}
      \pf_\tau(g^{-1}\up{\tau^*}g) = \nrd(g)^{-1} = c^{-1} = c
   \end{equation*}
   (see Lemma \ref{lem:pfaffian}).
  This proves the commutativity of the first rectangle.

  Finally, note that $\delta$ is injective since $\pf_\tau \circ \delta$ is injective.
\end{proof}
\begin{corollary}\label{cor:imageJ} An element
   $x \in H^1(\tau^*, \GL_\Lambda(C))$ lies in the image of $j_*$ if and only if
  $\pf_\tau(x)$ lies in the image of the canonical map $C^{(2)} \to C^\times/C^\times_\Lambda$.
\end{corollary}
\begin{proof}
   Let  $\alpha: C^{(2)} \to C^\times/C^\times_\Lambda$ denote the canonical map.
   Suppose the class ${x \in H^1(\tau^*, \GL_\Lambda(C))}$ is in the image of $j_*$, then we obtain immediately that $\pf_\tau(x)$ lies
   in the image of $\alpha$.

   Conversely, suppose $\pf_\tau(x) = \alpha(u)$ for some $u \in C^{(2)}$. 
   Then the diagram shows that $\nrd_{\Lambda/R}(x) = 1$ in $C^\times/(C_\Lambda^{\times})^2$
   and therefore $x$ lies in the image of $j_*$.
\end{proof}

\begin{remark}[Twisting involutions]\label{rem:twistingInvolutions}
 Let $A$ be a central simple $k$-algebra with an involution $\tau$ of symplectic type and let $\Lambda$ be a $\tau$-stable $R$-order.

Given an element $b \in \Sym(\Lambda, \tau) \cap \Lambda^{\times}$, we can
twist the involution $\tau$ with $b$. More precisely, we define $\tau|b: A \to A$ by $x \mapsto b \up{\tau}x b^{-1}$.
It is easily verified that this is again an involution on $A$, and since $b\in \Lambda^\times$, the order $\Lambda$ is $\tau|b$-stable.
Note that $\tau|b$ is again an involution of symplectic type.

Suppose $\Lambda$ is $\tau$-smooth, we claim that $\Lambda$ is $\tau|b$-smooth as well.
Take an element ${y \in \Sym(\Lambda,\tau|b)}$, this is $y = b \up{\tau}y b^{-1}$.
Consequently, $yb \in \Sym(\Lambda, \tau)$ and by \mbox{$\tau$-smoothness} there is an element $z \in \Lambda$ which satisfies
$\up{\tau}z + z = yb$. The element $b$ is a unit in $\Lambda$, hence we may write $z = wb$ for $w = zb^{-1} \in \Lambda$ and
it follows that $\up{\tau|b}w + w = y$. We have shown that $\Lambda$ is $\tau|b$-smooth.

Finally, for all $b \in \Sym(\Lambda, \tau) \cap \Lambda^{\times}$ we have
$(\tau|b)^* = \inn(b) \circ \tau^*$ on the group scheme $\GL_\Lambda$. Since $b = \up{\tau}b$ is equivalent to $b \up{\tau^*}b = 1$, such an
element $b$ is a cocycle for $H^1(\tau^*,\Lambda^\times)$. If we now twist $\tau^*$ with the cocycle $b$ (cf.~Section \ref{sec:RohlfsMethod}),
 we obtain 
\begin{equation*}
 \tau^*|b := \inn(b) \circ \tau^* = (\tau|b)^*.
\end{equation*}
\end{remark}

\subsection{Hermitian forms and non-abelian Galois cohomology}
We shall also need a result due to Fainsilber and Morales from the theory of hermitian forms.
Let $A$ be a central simple $k$-algebra and let $\tau$ be an involution on $A$. In this short section it is not important whether or not 
$\tau$ is of the first or of the second kind.

The notion of $\tau$-smoothness is related to the theory of even hermitian forms. 
Let $\Lambda$ be a $\tau$-stable $R$-order in $A$ and let $M$ be a finitely generated and projective right $\Lambda$-module.
A hermitian form $h$ (or more precisely a $1$-hermitian form) with respect to $\tau$ on $M$ is said to be \emph{even}
if there is a $\tau$-sesquilinear form $s: M \times M \to \Lambda$ such that $h = s + s^*$. Here $s^*$ is the sesquilinear form
 defined by 
\begin{equation*}
    s^*(x,y) := \up{\tau}s(y,x).
\end{equation*}
It follows immediately that $\Lambda$ is $\tau$-smooth if and only if every hermitian form on $\Lambda$ (considered as right $\Lambda$-module)
is even. This is useful since even hermitian forms can be handled easier than arbitrary hermitian forms.

We consider the automorphism $\tau^*$ of $\Lambda^\times$ defined as the composition of $\tau$ and the group inversion.
Similarly we obtain $\tau^*$ on $A^\times$.
Here it is not necessary to consider $\tau^*$ as a morphism of group schemes, which is a little bit more tedious if $\tau$ is of the second kind.
We will need a Theorem of Fainsilber-Morales in the following paraphrase:
\begin{theorem}[Fainsilber, Morales \cite{FainsilberMorales1999}]\label{thm:FainsilberMorales}
  Let $k$ be a field which is complete for a discrete valuation and let $R$ be its valuation ring.
  Let $A$ be a central simple $k$-algebra with involution $\tau$.
  Suppose $\Lambda$ is a $\tau$-stable maximal $R$-order in $A$. If $\Lambda$ is $\tau$-smooth, then
  the canonical map
  \begin{equation*}
      j_*: H^1(\tau^*, \Lambda^{\times}) \to H^1(\tau^*, A^\times)
  \end{equation*}
  is injective.
\end{theorem}
Compared with \cite{FainsilberMorales1999} we add the assumption of $\tau$-smoothness to eliminate the restriction on the residual characteristic. The proof is almost identical.

\section{An adelic reformulation of Harder's Gau\ss-Bonnet Theorem}\label{sec:HardersGB}
We briefly describe an adelic reformulation of Harder's Gau\ss-Bonnet Theorem (see \cite{Harder1971}) which hinges on the notion of smooth group scheme.

Let $F$ be an algebraic number field and let $\calO$ denote its ring of integers. Let $G$ be a connected semisimple algebraic group defined over $F$.
We denote by $G_\infty$ the associated real Lie group, i.e.
\begin{equation*}
   G_\infty = G(F \otimes_\bbQ \bbR) =\prod_{v \in V_\infty} G(F_v).
\end{equation*}
The group $G_\infty$ is a real semisimple Lie group.

\subsection{The Euler-Poincar\'e measure}
We define what we mean by the \emph{compact dual group} of $G_\infty$, since the definition differs from author to author.
Let $\LG_\infty$ be the real Lie algebra of $G_\infty$ and let $\LG_{\infty,\bbC}$ denote its complexification.
Moreover, let $K_\infty$ be a maximal compact subgroup of $G_\infty$ and consider the associated Cartan decomposition
\begin{equation*}
    \LG_\infty = \LK_\infty \oplus \LP.
\end{equation*}
The real vector space $\LU:= \LK_\infty \oplus i \LP \subseteq \LG_{\infty,\bbC}$ is a real Lie subalgebra of $\LG_{\infty,\bbC}$ 
and is even a compact real form of $\LG_{\infty,\bbC}$ (cf.~\cite[p.~360]{Knapp2002}).
Let $G_u$ be the unique connected (a~priori virtual) Lie subgroup of $G(F\otimes_\bbQ\bbC)$ with Lie algebra $\LU$.
Since the real semisimple Lie algebra $\LU$ is a compact form, the Lie group $G_u$ is compact and thus closed in $G(F\otimes_\bbQ\bbC)$
(see IV, Thm.~4.69 in \cite{Knapp2002}).
Further we see that the connected component $K^0_\infty$ is a subgroup of $G_u$.
We say that $G_u$ is the \emph{compact dual group} of $G_\infty$ containing $K^0_\infty$. Note that the dual group depends on the algebraic group $G$.

Let $B: \LG_\infty \times \LG_\infty \to \bbR$ be a non-degenerate $\bbR$-bilinear form such that $\LK_\infty$ and $\LP$ are orthogonal.
We extend $B$ the a $\bbC$-bilinear form (again denoted $B$) on $\LG_{\infty,\bbC}$. Note that $B$ restricted to $\LU$ is a non-degenerate $\bbR$-bilinear form
$\LU \times \LU \to \bbR$. We obtain corresponding right invariant volume densities on $G_\infty$ and on $G_u$ which will be denoted by~$\vold{B}$.

We define $X := \bsl{K_\infty}{G_\infty}$. Let $\Gamma \subseteq G(F)$ be a torsion-free arithmetic group. Harder's Gau\ss-Bonnet Theorem shows that integration
over $G_\infty/ \Gamma$ with the Euler-Poincar\'e measure $\mu_\chi$ (cf.~Serre \cite[§3]{Serre1971}) yields the Euler characteristic of $\Gamma$ -- even if $\Gamma$ is not cocompact.
Via Hirzebruch's proportionality principle one has the following
formula for the Euler-Poincar\'e measure on $G_\infty$ (cf. Harder \cite{Harder1971} and Serre \cite{Serre1971}).
\begin{theorem}\label{thm:EulerPoincareMeasure}
   If $\dim(X)$ is odd or if $\rk(\LK_{\infty,\bbC}) < \rk(\LG_{\infty,\bbC})$, then $\mu_\chi = 0$ is the Euler-Poincar\'e measure.
   Otherwise, if $\rk(\LG_{\infty,\bbC}) = \rk(\LK_{\infty,\bbC})$ and $\dim(X) = 2p$ is even, then 
   \begin{equation*}
      \mu_\chi := \frac{(-1)^p \left|W(\LG_{\infty,\bbC})\right|}{\left|\pi_0(G_\infty)\right| \: \left|W(\LK_{\infty,\bbC})\right|} \vol_B(G_u)^{-1} \vold{B}.
   \end{equation*}
   Here  $\pi_0(G_\infty) = G_\infty/G_\infty^0$ and $W(\LG_{\infty,\bbC})$ (resp.~$W(\LK_{\infty,\bbC})$) denotes the Weyl group of the complexified Lie algebra $\LG_{\infty,\bbC}$ (resp.~$\LK_{\infty,\bbC}$).

\end{theorem}

\subsection{The adelic reformulation}
 Let $\bbA$ denote the ring of adeles of $F$ and let $\bbA_f$ denote the ring of finite adeles.
 Let $G$ be a connected semisimple algebraic group defined over $F$.
Let $K_f \subseteq G(\bbA_f)$ be an open compact subgroup of the locally compact group $G(\bbA_f)$.
 Borel showed that $G(\bbA)$ is the disjoint union of a finite number $m$ of double cosets, i.e.
  \begin{equation*}
       G(\bbA) = \bigsqcup_{i =1}^m G_\infty K_f x_i G(F)
  \end{equation*}
 for some representatives $x_1, \dots, x_m \in G(\bbA_f)$ (see Thm.~5.1 in \cite{Borel1963}).
  For every $i=1,\dots, m$ we obtain an arithmetic subgroup $\Gamma_i \subseteq G(F)$ defined by
  \begin{equation*}
     \Gamma_i := G(F) \cap x_i^{-1} K_f x_i. 
  \end{equation*}
  There is a $G_\infty$-equivariant homeomorphism
   \begin{equation}\label{eq:GinvHomeo}
       \bsl{K_f}{G(\bbA)}/G(F)  \stackrel{\simeq}{\longrightarrow} \bigsqcup_{i=1}^m G_\infty/\Gamma_i.
   \end{equation}
   Here the right hand side denotes the topologically disjoint union.
   
   \begin{remark}\label{rem:NoComplexPlaces}
    Define $X = \bsl{K_\infty}{G_\infty}$.
    Suppose that
    $G(F)$ acts freely on $\bsl{K_\infty K_f}{G(\bbA)}$. 
    This is the case if and only if
    the groups $\Gamma_i$ are torsion-free for all $i = 1, \dots, m$.  
    If $\dim(X)$ is odd or if $\rk(\LK_{\infty,\bbC})<\rk(\LG_{\infty,\bbC})$, then
   \begin{equation*}
     \chi(\bsl{K_\infty K_f}{G(\bbA)}/G(F)) = 0.
    \end{equation*}
   This follows immediately from Harder's Gau\ss-Bonnet Theorem and the homeomorphism in equation \eqref{eq:GinvHomeo}.
   
   Note further that if $F$ has a complex place, then $\rk(\LK_{\infty,\bbC})<\rk(\LG_{\infty,\bbC})$ is always satisfied.
   Therefore we may restrict to the case where $F$ is totally real.
  \end{remark}
   
   \subsubsection{The Tamagawa measure}
   We derive a description of the Tamagawa measure in terms of the local volume densities. For a thorough definition of the Tamagawa measure we refer the reader to Oesterle
   \cite{Oesterle1984}.
   Let $G$ be a connected semisimple linear algebraic $F$-group of dimension $d$.
   Let $\LG = \Lie(G)(F)$ be the Lie algebra of $G$ over $F$. 
   
    Fix a non-degenerate $F$-bilinear form $B: \LG \times \LG \to F$ on the Lie algebra.
    For every place $v \in V$
    we have the left invariant
    volume density $\vold{B}_v$ attached to $B$ on the $F_v$-analytic manifold $G(F_v)$.
    More precisely, the volume density is uniquely determined by the relation $\vold{B}(e_1\wedge \dots \wedge e_d) = |\det(B(e_i,e_j))|^{1/2}$ for all 
    $e_1, \dots, e_d \in \LG$.
    
    We fix Haar measures $\mu_v$ on $F_v$ for every place $v$ such that
   \begin{itemize}
     \item  $\mu_v(\calO_v)=1$ if $v\in V_f$ is a finite place,
     \item  $\mu_v([0,1]) = 1$ if $v$ is a real place, and
     \item  $\mu_v([0,1]+[0,1]i) = 2$ if $v$ is a complex place.
   \end{itemize}
   Using these choices of Haar measures a density on $G(F_v)$ defines a measure on the analytic manifold $G(F_v)$.
   
\begin{lemma}\label{lem:TamagawaVolume}
  Let $G$ be a $d$-dimensional semisimple connected linear algebraic group defined over $F$. 
  Fix a non-degenerate $F$-bilinear form $B: \LG \times \LG \to F$ on the Lie algebra.
  Then the Tamagawa measure on $G(\bbA)$ is given by
  \begin{equation*}
     \tau = \disc{F}^{-d/2} \prod_{v \in V} \vold{B}_v.
  \end{equation*}
\end{lemma}
\begin{proof}
  Let $e_1, \dots, e_d$ be a basis of $\LG$ over $F$ and take
  the dual basis $\varepsilon_1,\dots, \varepsilon_d$ of $\Hom_F(\LG,F)$. We define a non-trivial
  form of highest degree $\omega = \varepsilon_1\wedge \dots \wedge \varepsilon_d$ on~$\LG$.
  By definition of the volume density we have 
  \begin{equation*}
    \vold{B}_v = |\det(B(e_i,e_j))|^{1/2}_v \dom{v}.
  \end{equation*}
  By the product formula we know that $|\det(B(e_i,e_j))|_v = 1$ for almost all places $v$ and further that $\prod_{v \in V} |\det(B(e_i,e_j))|_v = 1$.
\end{proof}

\subsubsection{The modulus factor}\label{par:modulusFactor}
 We focus on the case where the algebraic group has a smooth $\calO$-model.
 Let $G$ be a smooth group scheme defined over $\calO$. 
 For any commutative $\calO$-algebra $R$ we write $\LG_R := \Lie(G)(R)$ to denote the $R$-points of the Lie algebra of $G$.
 Let $B: \LG_F \times \LG_F \to F$ be a non-degenerate $F$-bilinear form. For every finite place $v \in V_f$ we
 define the \emph{modulus factor} $\mf(B)_v$ as follows:
 take an $\calO_v$-basis $e_1, \dots, e_n$ of the free $\calO_v$-module $\LG_v := \LG_{\calO_v}$, and define
  \begin{equation*}
      \mf(B)_v := \left|\det(B(e_i,e_j))\right|^{1/2}_v.
  \end{equation*}
 For almost all finite places $v \in V_f$ we have $\mf(B)_v = 1$. To see this, take an $F$-basis of $\LG_F$ and note that it is
 an $\calO_v$-basis of $\LG_{v}$ for almost all finite places $v$. 
 This allows us to define the \emph{global modulus factor} $\mf(B) := \prod_{v \in V_f} \mf(B)_v$.

\subsubsection{Congruence groups}\label{par:congruenceGroups}
In the adelic formulation of Harder's Gau\ss-Bonnet Theorem we focus on congruence groups which are given by local data.
Let $G$ be a smooth $\calO$-group scheme. For every finite place $v \in V_f$, let $\alpha_v \geq 1$ be a natural number and we assume that $\alpha_v = 1$
for almost all $v \in V_f$. Let $v$ be a finite place and let $\LP_v \subseteq \calO_v$ be the unique prime ideal in $\calO_v$. 
 We define $\pi_v$ to be the reduction morphism
 \begin{equation*}
     \pi_v: G(\calO_v) \to G(\calO_v/\LP^{\alpha_v}_v).
 \end{equation*}
 Further, we assume that we are given a subgroup $U_v$ of the finite group $G(\calO_v/\LP_v^{\alpha_v})$ for every place $v \in V_f$.
 For a place $v \in V_f$ the group $K_v(U) := \pi_v^{-1}(U_v)$ is an open compact subgroup of $G(\calO_v)$.
 If we additionally impose the assumption that $U_v = G(\calO_v/\LP_v^{\alpha_v})$ for almost all $v$,
 then the group
  \begin{equation*}
     K(U) := \prod_{v \in V_f} K_v(U)
  \end{equation*}
  is an open compact subgroup of the locally compact group $G(\bbA_f)$.
  We say that $K(U)$ is the congruence group associated with the \emph{local datum}\index{local datum}
  \begin{equation*}
    U = (U_v)_v = (U_v, \alpha_v)_v 
  \end{equation*}
 (usually the numbers $\alpha_v$ are considered to be implicitly a part of the datum $U$).
 
\subsubsection{The adelic Euler characteristic formula}
  Let $F$ be a totally real number field. Let $G$ be a smooth
  group scheme over $\calO$ such that $G \times_\calO F$ is a connected semisimple group.
  For every real place $v$ we choose a maximal compact subgroup $K_v \subseteq G(F_v)$.
  The real Lie algebra of $K_v$ will be denoted $\LK_v$.
  The product $K_\infty = \prod_{v \in V_\infty} K_v$ is a maximal compact subgroup of the associated real Lie group $G_\infty$. We denote the Lie algebra of $K_\infty$ by $\LK$. 
  
  Let $B : \LG_F \times \LG_F \to F$ be an $F$-bilinear form. We say that $B$ is \emph{nice} with respect to $K_\infty$
  if $B$ is non-degenerate and for every real place $v \in V_\infty$ the Cartan decomposition w.r.t.~$\LK_v$ is orthogonal with respect to $B$.
  A nice form induces a non-degenerate bilinear form $B: \LG_\infty \times \LG_\infty \to \bbR$ by defining the Lie subalgebras $\LG_v = \Lie(G(F_v))$ to 
  be orthogonal. Note that the form $B$ satisfies the requirements of Theorem \ref{thm:EulerPoincareMeasure}.
\begin{theorem}\label{thm:adelicFormula}
   Let $G$ be a smooth
  group scheme over $\calO$ such that $G \times_\calO F$ is a connected semisimple group of dimension $d$.
 We fix any nice form $B : \LG_F \times \LG_F \to F$.
 Furthermore, let ${K_f = K(U)}$ be a congruence subgroup of $G(\bbA_f)$ given by a local datum $(U, \alpha)$ such that
  $G(F)$ acts freely on $\bsl{K_\infty K_f}{G(\bbA)}$. 

  If $\dim(X) = 2p$ is even and $\rk(\LK_\bbC) = \rk(\LG_{\infty,\bbC})$, then
  the Euler characteristic of  $\bsl{K_\infty K_f}{G(\bbA)}/G(F)$ is given by
  \begin{align*}
     \chi(\bsl{K_\infty &K_f}{G(\bbA)}/G(F)) \\ 
 &=(-1)^p \disc{F}^{d/2} \frac{\left|W(\LG_{\infty,\bbC})\right| \:\tau(G)}{\left|\pi_0(G_\infty)\right| \left|W(\LK_\bbC\right)|}
   \vol_B(G_u)^{-1} \mf(B)^{-1} \prod_{v \in V_f} \frac{\N(\LP_v)^{d\alpha_v}}{|U_v|}  .
  \end{align*}
 Here $\tau(G)$ is the Tamagawa number of $G$, $\N(\LP_v)$ denotes the cardinality of the residue class field $\calO_v/\LP_v$, and
 $G_u$ denotes the compact dual group of $G^0_\infty$
 (the remaining notation is as in Theorem \ref{thm:EulerPoincareMeasure}).
\end{theorem}
\begin{proof}
 Let $x_1, \dots, x_m \in G(\bbA_f)$ be representatives of the finitely many double
 cosets in $\bsl{G_\infty K_f}{G(\bbA)}/G(F)$.
 We consider the torsion-free arithmetic groups $\Gamma_i$ defined as $\Gamma_i := G(F) \cap x_i^{-1}K_fx_i$.
 Let $\calF_i$ be a Borel measurable fundamental domain for the right action of $\Gamma_i$ on $G_\infty$.
 Here we mean a fundamental domain in the strict sense, i.e.~$\calF_i$ is a set of representatives for $G_\infty/\Gamma_i$
 (for the existence of measurable fundamental domains see Bourbaki -- Int\'egration, VII.§2 Ex.~12 \cite{BourbakiInt}).
 The set $\calF$ defined as the union $\bigsqcup_{i=1}^m \calF_iK_fx_i \subseteq G(\bbA)$ is a Borel measurable
 fundamental domain for the right action of $G(F)$ on $G(\bbA)$.
 We write $\vold{B}_\infty = \prod_{v \in V_\infty} \vold{B}_v$ and $\vold{B}_f := \prod_{v\in V_f} \vold{B}_v$.
 Due to Theorem \ref{thm:EulerPoincareMeasure} we have
 \begin{equation*}
      \chi(\bsl{K_\infty K_f}{G(\bbA)}/G(F)) = \sum_{i=1}^m \chi(X/\Gamma_i) = \lambda \sum_{i = 1}^m \int_{\calF_i} \vold{B}_\infty
 \end{equation*}
  where $\lambda = (-1)^p \frac{\left|W(\LG_{\infty,\bbC})\right|}{\left|\pi_0(G_\infty)\right| \left|W(\LK_\bbC\right)|}
   \vol_B(G_u)^{-1}$.
  By multiplication with the volume of $K_f$, which is simply $\vol_B(K_f) =\int_{K_f} \vold{B}_f$, and by Lemma \ref{lem:TamagawaVolume}, we obtain
   \begin{equation*}
      \sum_{i = 1}^m \int_{\calF_i} \vold{B}_\infty \vol_B(K_f) = \int_{\calF} \prod_{v \in V}\vold{B}_v = \disc{F}^{d/2}\int_{\calF} \tau = \disc{F}^{d/2}\tau(G).
 \end{equation*}
 This means, we have 
  \begin{equation*}
      \chi(\bsl{K_\infty K_f}{G(\bbA)}/G(F)) = \lambda \disc{F}^{d/2}\tau(G)  \vol_B(K_f)^{-1}.
 \end{equation*}
 Finally we are left with the task of determining $\vol_B(K_f)$.
 We shall exploit that $K_f$ is given by the local datum $(U, \alpha)$.
 Since $\vol_B(K_f) = \prod_{v \in V_f} \vol_B(K_v(U))$ and the scheme $G$ is smooth, we can apply a Theorem of Weil (for a modern formulation see Prop.~2.5 in \cite{Oesterle1984} or Thm.~2.5 in \cite{Batyrev1999}) 
 in every finite place to deduce
 \begin{equation*}
   \vol_B(K_f) = \prod_{v \in V_f} \mf(B)_v \frac{|U_v|}{\N(\LP_v)^{d\alpha_v}}.
 \end{equation*}
 Now the claim follows readily.
\end{proof}

\section{Rohlfs' method}\label{sec:RohlfsMethod}
In this section we give a short summary of Rohlfs' method for the computation of Lefschetz numbers. 

Let $F$ be an algebraic number field and let $G$ be a linear algebraic group defined over~$F$. We assume that $G$ has strong approximation.
For example, unipotent groups and $F$-simple, simply connected groups with a non-compact associated Lie group
have strong approximation (see p.~427 in \cite{PlatonovRapinchuk1994}).
Choose a maximal compact subgroup $K_{\infty} \subseteq G_\infty$ 
and set $X := \bsl{K_{\infty}}{G_\infty}$ .
Furthermore, let $K_f \subseteq G(\bbA_f)$ be an open compact subgroup and let $\Gamma := G(F) \cap K_f$ be the arithmetic group 
defined by this open compact subgroup. There is a homeomorphism
\begin{equation*}
      X/\Gamma \: \stackrel{\simeq}{\longrightarrow}\: \bsl{K_{\infty}K_f}{G(\bbA)}/G(F).
\end{equation*}
To see this, consider with the inclusion $G_\infty \to G(\bbA)$ and use strong approximation to
observe that it factors to such a homeomorphism.
Recall that $\Gamma$ is torsion-free if and only in $G(F)$ acts freely on $\bsl{K_{\infty}K_f}{G(\bbA)}$.

Let $\tau$ be an automorphism of finite order of $G$. We can choose $K_\infty$ such that it is $\tau$-stable.
We further assume that $K_f \subset G(\bbA_f)$ is a $\tau$-stable open compact subgroup.
We obtain an action of $\tau$ on the double coset space 
\begin{equation*}
S(K_f) := \bsl{K_{\infty}K_f}{G(\bbA)}/G(F).
\end{equation*}
We describe the set $S(K_f)^{\tau}$ of $\tau$-fixed points following Rohlfs \cite{Rohlfs1990} under 
the assumption that $G(F)$ acts freely on $\bsl{K_{\infty}K_f}{G(\bbA)}$.

Consider the finite set $\calH^1(\tau)$ defined as the fibred product
\begin{equation*}
    \calH^1(\tau) := H^1(\tau,K_\infty K_f) \mathop{\times}_{H^1(\tau,G(\bbA))} H^1(\tau, G(F)).
\end{equation*}
of non-abelian cohomology sets. 
Here we usually write $\tau$ instead of the finite group $\langle\tau\rangle$ generated by $\tau$.
We consider $\calH^1(\tau)$ as a topological space with the discrete topology.
Rohlfs constructed (cf.~3.5 in \cite{Rohlfs1990}) a surjective and continuous map
\begin{equation*}
  \vartheta: S(K_f)^{\tau} \to \calH^1(\tau).
\end{equation*}
In particular the fibres are open and closed in $S(K_f)^{\tau}$ and we get a decomposition
\begin{equation}\label{eqAdelicDecomposition}
    S(K_f)^{\tau} \:=\: \bigsqcup_{\eta \in \calH^1(\tau)} \vartheta^{-1}(\eta).
\end{equation}

Let $\gamma \in Z^1(\tau, G(F))$ be a cocycle. The $\gamma$-twisted $\tau$-action
on $G$, defined by $\up{\tau|\gamma}(x) = \gamma_\tau \up{\tau}x \gamma_\tau^{-1}$, is an automorphism defined over $F$ and the group of fixed points is a linear algebraic group which will be denoted
$G(\gamma)$. Similarly, given a cocycle $k \in Z^1(\tau, K_\infty K_f)$ we define the $k$-twisted action 
of $\tau$ on $K_\infty K_f$ by $\up{\tau|k}g := k_\tau \up{\tau}g k_\tau^{-1}$. The corresponding group of fixed points under this action
will be written $(K_\infty K_f)^{\tau|k}$.
Rohlfs obtained the following description of the fibres of $\vartheta$.

\begin{lemma}[Rohlfs, see 3.5 in \cite{Rohlfs1990}]\label{lem:StructureFPComponents}
  Let $K_f \subset G(\bbA_f)$ be a $\tau$-stable open compact subgroup such that $G(F)$ acts freely on 
  $\bsl{K_\infty K_f}{G(\bbA)}$.

  Let $\eta \in \calH^1(\tau)$ be a class represented by a pair of cocycles $(k,\gamma)$ with
  $(k_s)_s$ in $Z^1(\tau, K_\infty K_f)$ and $(\gamma_s)_s \in Z^1(\tau, G(F))$.
  Take $a \in G(\bbA)$ such that $\up{s}a = k^{-1}_s a \gamma_s$ for all $s \in \langle\tau\rangle$.
  There is a homeomorphism 
\begin{equation*}
     \bsl{a^{-1} (K_\infty K_f)^{\tau|k}a}{G(\gamma)(\bbA)}/G(\gamma)(F) \: \stackrel{\simeq}{\longrightarrow} \:  \vartheta^{-1}(\eta).
\end{equation*}
\end{lemma}

Combined with the next theorem this yields a method for the computation of Lefschetz numbers which we simply call \emph{Rohlfs' method}.
\begin{definition}\label{def:compatibleAction}
 Let $\rho: G \to \GL(W)$ be a rational representation defined over the algebraic closure $\alg{F}$ of $F$. Here $W$ is a
 finite dimensional $\alg{F}$-vector space.
 Given an action of the finite group $\langle \tau \rangle$ on $W$.
 We say that this action is \emph{compatible} with $\rho$, if 
\begin{equation*}
   \up{s}(\rho(g)v) \: = \: \rho(\up{s}g) \up{s}v
\end{equation*}
for all $v \in V$, $s \in \langle\tau\rangle$ and $g \in G(\alg{F})$.
In other words $W$ is a $(G(\alg{F})\rtimes\langle\tau\rangle)$-module.
\end{definition}
Let $\rho: G \to \GL(W)$ be a rational representation and let $\Gamma \subseteq G(F)$ be a torsion-free arithmetic subgroup.
If $W$ is equipped with a compatible $\tau$-action then
we define the \emph{Lefschetz number} of $\tau$ with values in $W$ as
\begin{equation*}
   \calL(\tau, \Gamma, W) := \sum_{i=0}^\infty (-1)^i \Tr\bigl(\tau^i: H^i(\Gamma, W) \to H^i(\Gamma, W)\bigr).
\end{equation*}
Since torsion-free arithmetic groups are of type (FL), this is a finite sum.

Given a cocycle $b=(b_s)_s \in H^1(\tau, G(\alg{F}))$ one can define the $b$-twisted $\tau$-action on $W$ by
\begin{equation*}
    \up{\tau|b}w = b_\tau \up{\tau}w
\end{equation*}
for all $w \in W$. We write $W(b)$ to denote the space $W$ with the $b$-twisted $\tau$-action.
We need the following slight paraphrase of a theorem of Rohlfs.
\begin{theorem}[cf.~Rohlfs \cite{Rohlfs1990}]\label{thmLefschetznumberRationalRepAdelic}
 Let $G$ be an algebraic $F$-group with strong approximation and let $\tau$ be an automorphism of finite order defined over $F$.
 Let $K_f \subset G(\bbA_f)$ be a $\tau$-stable open compact subgroup 
such that $\Gamma := G(F) \cap K_f$ is torsion-free. Let $\rho: G \to \GL(W)$ be a rational representation defined over $\alg{F}$ with
a compatible $\tau$-action. Then we have
\begin{equation*}
     \calL(\tau,\Gamma,W)  \: = \: \sum_{\eta \in \calH^1(\tau)} \chi\bigl(\vartheta^{-1}(\eta)\bigr) \Tr(\tau | W(b_\eta)),
\end{equation*}
where $b_\eta \in G(F)$ is any representative of the $H^1(\tau, G(F))$ component of $\eta$.
\end{theorem}
\begin{proof}
  This follows from Rohlfs' decomposition (see equation \eqref{eqAdelicDecomposition}) 
  and a suitable Lefschetz fixed point principle, for instance \cite{Kionke2012}.
  %%% TODO: More precise?
\end{proof}

\section{Proof of the main theorem}\label{sec:ProofMain}

\subsection{Introduction}
In this section
we compute the Lefschetz number of an involution of symplectic type on principal congruence subgroups
of inner forms of the special linear group.
For this purpose we combine the tools developed in the previous sections.

One should keep in mind that the central result is the adelic Lefschetz number formula in Theorem \ref{thmLefschetznumberRationalRepAdelic}.
Whenever we want to apply this theorem, there are two important steps to do. 
First step:
understand the involved first non-abelian Galois cohomology sets.
Second step: compute the Euler characteristics
of the fixed point components.
In the second step we use the adelic formula in Theorem \ref{thm:adelicFormula} obtained from Harder's Gau\ss-Bonnet Theorem.

First we introduce some notation,
then we begin to determine various non-abelian cohomology sets.
In the third subsection we describe the fixed point groups and we compute their Euler characteristics.
Finally we prove the main theorem.

As before $F$ denotes an algebraic number field and $\calO$ denotes its ring of integers.
Let $D$ be a quaternion algebra over $F$, i.e.~a central simple $F$-algebra of dimension four.
 Note that, even though we use the symbol $D$, the quaternion algebra $D$ is in general not assumed to be a division algebra.
 Given a place $v$, we define $D_v := D \otimes_F F_v$. If $D_v$ is isomorphic to $M_2(F_v)$, we say that $D$ \emph{splits} at the place $v$.
 Otherwise $D_v$ is a division algebra and we say that $D$ is \emph{ramified} at $v$.
 Let $\Ram(D) \subset V$ be the finite set places where $D$ ramifies, and let $\Ram_f(D)$ (resp.~$\Ram_\infty(D)$) denote the subset 
 of finite (resp.~archimedean) places.

\begin{definition}\label{def:redDiscr}
  The \emph{signed reduced discriminant} $\Drd{D}$ of $D$ is the integer
  \begin{equation*}
     \Drd{D} := (-1)^r\prod_{\LP \in \Ram_f(D)} \N(\LP).
  \end{equation*}
  where $r = |\Ram_\infty(D)|$.
\end{definition}

\subsubsection{The canonical involution}\label{par:involution}
 On the quaternion algebra $D$ we have the \emph{canonical involution}, sometimes called \emph{conjugation},
 \begin{equation*}
    \tau_c: D \to D \:\text{ denoted by }\: x \mapsto \conj{x}.
 \end{equation*}
  Given a description as $D = Q(a,b | F)$ with $a,b \in F^\times$, i.e.~there is a basis $1, i, j, ij$ of $D$ with $i^2 = a$, 
  $j^2 = b$ and $ij = -ji$, then the conjugation is defined by
  \begin{equation*}
      \tau_c: x_0 + x_1 i + x_2 j + x_3 ij \mapsto x_0 - x_1 i - x_2 j - x_3 ij.
  \end{equation*}
  Note that the conjugation is $F$-linear, i.e.~it is an involution of the first kind on $D$.
  Moreover, $\tau_c$ is an involution of symplectic type. 

  The elements fixed by conjugation are precisely the elements of $F$.
  The conjugation is related to the reduced norm and trace of $D$ by
  \begin{align*}
     \trd_D(x) &= x + \conj{x}, \\
     \nrd_D(x) &= x \conj{x} = \conj{x} x
  \end{align*}
  for all $x$ in $D$.

\subsubsection{Orders}
  Let $\Lambda_D$ be an $\calO$-order in $D$. 
  Then $\Lambda_D$ is $\tau_c$-stable, as can be seen as follows:
  Let $x \in \Lambda_D$, then
  \begin{equation*}
     \conj{x} = x + \conj{x} - x = \trd_D(x) - x.
  \end{equation*}
  Recall that $\trd_D(x) \in \calO$ because $x$ is integral. Since $\calO \subseteq \Lambda_D$, we obtain $\conj{x} \in \Lambda_D$.
  Moreover,  it follows directly from the definitions that $\Lambda_D$ is smooth
  if and only if $\Lambda_D$ is $\tau_c$-smooth (see Def.~\ref{def:smooth} and Def.~\ref{def:tausmooth}).

  We will assume from now on that $\Lambda_D$ is a maximal $\calO$-order in $D$. In particular,
  it is a smooth and $\tau_c$-smooth order (see Prop.~\ref{prop:maxImpliesSmooth}).
  
  Let $n$ be a positive integer. Consider the central simple $F$-algebra
  \begin{equation*}
      A := M_n(D)
  \end{equation*}
  of $n \times n$-matrices with entries in the quaternion algebra $D$. 
  The canonical involution on $D$ induces an involution $\tau$ on $A$ defined by
  \begin{equation*}
     \tau(x) := \up{\tau}x := \conj{x}^T,
  \end{equation*}
  i.e.~conjugate every entry in the matrix $x$ and then transpose the matrix. 
  It is easily checked that this defines an involution of symplectic type on $A$ (cf.~(2.23) in \cite{BOInv}).
  
\begin{lemma} \label{lem:orderIsNice}
  Let $\Lambda_D \subseteq D$ be a maximal $\calO$-order.
  The $\calO$-order $\Lambda = M_n(\Lambda_D)$ in $A$ is maximal, $\tau$-stable, smooth and $\tau$-smooth.
\end{lemma}
\begin{proof}
   Since $\Lambda_D$ is stable under conjugation, it is obvious that $\Lambda$ is $\tau$-stable.
   Moreover, it follows from (21.6) in \cite{Reiner2003} that $\Lambda$ is a maximal $\calO$-order. 
   In turn Proposition \ref{prop:maxImpliesSmooth} shows that $\Lambda$ is also a smooth order.

   Finally we need to check that $\Lambda$ is $\tau$-smooth. Let $x \in \Sym(\Lambda, \tau)$ be an element which is
   fixed by $\tau$. This means that $x = (x_{ij})$ satisfies 
   \begin{align*}
      x_{ij} &= \conj{x_{ji}} \: \text{ for all }\: i \neq j, \:\text{ and}\\
      x_{ii} &\in \calO .
   \end{align*}
   The order $\Lambda_D$ is smooth, therefore there is, for every $i=1, \dots, n$,
   an element $z_i \in \Lambda_D$ with $\trd_D(z_i) = z_i + \conj{z_i} = x_{ii}$.
   Now we define the upper triangular element $y \in \Lambda$ by
   \begin{equation*}
      y_{ij} := \begin{cases}
                   0  \quad &\text{ if $i > j$}\\
                   x_{ij}   & \text{ if $i < j$}\\
                   z_i      & \text{ if $i=j$},
               \end{cases}
   \end{equation*}
   and it is easy to see that $y + \up{\tau}y = x$. 
   We deduce that $\Lambda$ is $\tau$-smooth.
\end{proof}

\subsubsection{Setting and Assumptions} \label{par:defCongruenceGroups}
We define $G:= \SL_\Lambda$ to be the special linear group over the order $\Lambda$ (see Def.~\ref{def:SL}).
From the previous lemma and Cor.~\ref{cor:smoothSL} we deduce that $G$ is a smooth group scheme over $\calO$.
Moreover, the involution $\tau$ induces an 
automorphism $\tau^*$ of $\GL_\Lambda$ where $\tau^* = \inv \circ \tau$ (cf.~\ref{sec:FixedPointGroups}).
Clearly $\tau^*$ has order (at most) two and it restricts to an automorphism of $G = \SL_\Lambda$.
 
The real Lie group $G_\infty$ associated with $G$ is
\begin{equation*}
   G_\infty := \prod_{v \in V_\infty} G(F_v) \:\cong\: \SL_{2n}(\bbR)^s \times \SL_n(\bbH)^r \times \SL_{2n}(\bbC)^t. 
\end{equation*}
Here $s$ denotes the number real places of $F$ where $D$ splits, $r$ is the number of real places where $D$ ramifies,
and $t$ is the number of complex places of $F$. The symbol $\bbH$ is used for Hamilton's quaternion division algebra and
$\SL_n(\bbH)$ is the group of elements with reduced norm one in the central simple $\bbR$-algebra $M_n(\bbH)$.
Note that $[F:\bbQ] = r+s+2t$.
For every archimedean place $v$ we fix a $\tau^*$-stable maximal compact subgroup ${K_v \subseteq G(F_v)}$, then the group
 $K_\infty:= \prod_{v\in V_\infty} K_v$ is a $\tau^*$-stable maximal compact subgroup of $G_\infty$.

We study the cohomology of congruence subgroups arising from the group $\SL_\Lambda$.
Let $\LA \subseteq \calO$ be a proper ideal, we define the principal congruence subgroup
\begin{equation*}
 \Gamma(\LA) :=  \ker\bigl( G(\calO) \to G(\calO/\LA) \bigr)
\end{equation*}
of level $\LA$. We shall always assume that $\Gamma(\LA)$ is torsion-free (which holds for almost all ideals). Note that the groups $\Gamma(\LA)$ are always $\tau^*$-stable.

These groups can be described by local data. Let $\LP\subseteq \calO$ be a prime ideal of $\calO$ and let $v$ be the associated
finite place. Let $\nu_\LP(\LA)$ be the maximal exponent $e$ such that $\LP^e$ divides $\LA$, then $\LA\calO_v = \LP^{\nu_\LP(\LA)}\calO_v$.
We obtain an open and compact subgroup $K_v \subseteq G(\calO_v)$ defined as
 \begin{equation*}
     K_v := \ker\bigl( G(\calO_v) \longrightarrow G(\calO_v/\LA\calO_v) \bigr).
 \end{equation*}
We form the direct product $K_f := \prod_{v \in V_f}K_v$, this is an open and compact subgroup 
of the locally compact group $G(\bbA_f)$. Clearly, $\Gamma(\LA) = G(F) \cap K_f$.

We keep the notation introduced in this section. We always assume that
 \begin{enumerate}
  \item The order $\Lambda_D$ is a maximal order in $D$, and
  \item the ideal $\LA \subseteq \calO$ is non-trivial and chosen such that $\Gamma(\LA)$ is torsion-free.
 \end{enumerate}
 
\subsection{Hermitian forms and Galois cohomology}\label{sec:HermitianFormsAndCoho}
 In this section we determine the non-abelian Galois cohomology set $\calH^1(\tau^*)$. Recall that $\calH^1(\tau^*)$
 is the fibred product 
 \begin{equation*}
     \calH^1(\tau^*) := H^1(\tau^*,K_\infty K_f) \mathop{\times}_{H^1(\tau^*,G(\bbA))} H^1(\tau^*, G(F)).
 \end{equation*}
 In order to determine this set we need to calculate local and global cohomology sets.
 The global problem is to determine $H^1(\tau^*, G(F))$, whereas locally
 we have to calculate $H^1(\tau^*, G(F_v))$ and $H^1(\tau^*, K_v)$ for every place $v$.
 We start by determining the corresponding cohomology sets for $\GL_\Lambda$. 
 This task amounts to the classification of certain hermitian forms over quaternion algebras,
 which is basically well-known (see for instance \cite[§2]{Shimura1963} or \cite[Ch.~10]{Scharlau1985}).
 Afterwards we use the pfaffian to obtain results for the special linear group.
 
 \subsubsection{Local results for $\GL_\Lambda$}
 We introduce the following notation: Given two integers $p,q \geq 0$ with $p+q = n$, we
 define the diagonal matrix
 \begin{equation*}
    I_{p,q} = \diag(\underbrace{1,\dots,1}_{p},\underbrace{-1,\dots,-1}_{q}).
 \end{equation*}
 
\begin{proposition}\label{prop:H1GLlocalFields}
  Let $v \in V$ be a place of $F$. If $v$ is a real place where $D$ is ramified, then
  \begin{equation*}
    H^1(\tau^*, \GL_\Lambda(F_v)) \cong \{\:I_{p,q} \:|\: p,q \geq 0 \text { with } p+q = n \:\}.
  \end{equation*}
  This means that the matrices $I_{p,q}$ are a system of representatives for the cohomology classes.
  The cohomology is trivial for all places $v \in V \setminus \Ram_\infty(D)$ , i.e.
  \begin{equation*}
     H^1(\tau^*, \GL_\Lambda(F_v)) =\{1\}.
  \end{equation*}
\end{proposition}
\begin{proof}
  Let $b \in Z^1(\tau^*, \GL_\Lambda(F_v))$ be a cocycle, this is, $b$ is an element of $\GL_n(D_v)$ satisfying $b = \up{\tau}b$.
  Such a matrix $b$ defines a regular hermitian form on the free right $D_v$-module $D_v^n$.
  
  For all $v \in V \setminus \Ram_\infty(D)$, i.e.~$v$ is not a real ramified place,
  the regular hermitian forms over $D_v$ are classified by their dimension over $F_v$, this follows 
  from Ch.~10, Thm.~1.7 and Ex.~1.8 in \cite{Scharlau1985}. Note that these results cover the case where $D_v$ is a division algebra. However it is easy to obtain
  an analogous result if $D_v \cong M_2(F_v)$ (at least for free regular hermitian spaces).
  Thus we find $g \in \GL_n(D_v)$ with $gb\up{\tau}g = 1$, and so the second assertion follows immediately.

  Let $v \in \Ram_\infty(D)$, then $D_v \cong \bbH$. 
  In this case $\tau_c$-hermitian forms are classified by dimension and signature.
  Translated to the setting of non-abelian Galois cohomology, this means that
  the set ${\{\:I_{p,q} \:|\: p,q \geq 0 \text { with } p+q = n \:\}}$
  is a system of representatives for  $H^1(\tau^*, \GL_\Lambda(F_v))$.
\end{proof}

\begin{definition}\label{def:Signature}
  Let $v \in \Ram_\infty(D)$. For a cocycle $b \in Z^1(\tau^*, \GL_\Lambda(F_v))$ 
  which is cohomologous to $I_{p,q}$ we say that the \emph{signature}
  of $b$ is the pair $(p,q)$.
\end{definition}

\begin{corollary}\label{cor:H1localfiniteGL}
  Let $v \in V_f$ be a finite place, then $H^1(\tau^*, \GL_\Lambda(\calO_v)) = \{1\}$.
\end{corollary}
\begin{proof}
  The $\calO$-order $\Lambda$ is maximal and $\tau$-smooth (see \ref{lem:orderIsNice}) and the same holds for 
  the $\calO_v$-order $\Lambda \otimes \calO_v$ (cf.~(11.6) in \cite{Reiner2003} and note that $\tau$-smoothness is a local property).
  By the theorem of Fainsilber-Morales (Thm.~\ref{thm:FainsilberMorales}) the canonical map 
  \begin{equation*}
     H^1(\tau^*, \GL_\Lambda(\calO_v)) \to H^1(\tau^*, \GL_\Lambda(F_v))
  \end{equation*}
 is injective, and hence
  the assertion follows immediately from Proposition \ref{prop:H1GLlocalFields}. 
\end{proof}

\subsubsection{Global results for $\GL_\Lambda$}
Now we start to attack the global problem and analyse $H^1(\tau^*, \GL_\Lambda(F))$.
Again we use the classification of $\tau_c$-hermitian forms.

\begin{proposition}[Hasse principle]\label{prop:HassePrinciple}
    The canonical map 
   \begin{equation*}
         H^1(\tau^*, \GL_\Lambda(F)) \longrightarrow \prod_{v \in \Ram_\infty(D)} H^1(\tau^*, \GL_\Lambda(F_v))
   \end{equation*}
   induced by the inclusions is bijective. This
   means a class in $H^1(\tau^*, \GL_\Lambda(F))$ is uniquely determined by its signatures at the real ramified places.
\end{proposition}
\begin{proof}
  If $D$ is not a division algebra is is easily checked that $H^1(\tau^*, \GL_\Lambda(F)) = \{1\}$.
  Thus there is nothing to show.

  Assume that $D$ is a division algebra.
  The regular hermitian forms over $D$ (w.r.t.~$\tau_c$) are classified by dimension and their signatures at the real places of $F$ where $D$ ramifies (see Ch.~10, 1.8 in \cite{Scharlau1985}).
  The claim follows as in the local case.
\end{proof}

\subsubsection{The pfaffian associated with $\tau$}\label{par:pfaffianDiag}
We explain how to compute the pfaffian associated with $\tau$ (see \ref{par:pfaffian}) for diagonal matrices.
Let $k$ be any extension field of $F$, for example a local completion.
Given a diagonal matrix $x = \diag(x_1, \dots, x_n)$ with entries in $k$, we can consider
$x$ as a $\tau$-fixed matrix in $A \otimes_F k = M_n(D\otimes_F k)$.
\begin{lemma}
   For $x = \diag(x_1, \dots, x_n)$ with entries in some extension field $k$ of $F$, the pfaffian of $x$
   is the product of all entries, i.e.
   \begin{equation*}
      \pf_\tau(x) = x_1 x_2 \cdots x_n.
   \end{equation*}
\end{lemma}
\begin{proof}
    We can assume without loss of generality that $k$ is algebraically closed.
    In this case $D \otimes_F k \cong M_2(k)$ and the reduced norm $\nrd_D: D \otimes_F k \to k$ agrees with the determinant,
    in particular it is surjective.
    This means, for given $i \in \{1,\dots,n\}$, we can write $x_i = \nrd_D(y_i) = \conj{y_i}y_i$ for some $y_i \in D \otimes_F k$.
    Consider the matrix $y = \diag(y_1,\dots,y_n) \in M_n(D \otimes_F k)$, this matrix satisfies $\tau(y) y = x$.
    By Lemma \ref{lem:pfaffian} we obtain
    \begin{equation*}
        \pf_\tau(x) = \nrd_{A}(y) = \prod_{i=1}^n \nrd_D(y_i) = \prod_{i=1}^n x_i.
    \end{equation*}
    Here we used that the reduced norm of a diagonal matrix in $M_n(D \otimes_F k)$ is
    the product of the reduced norms of the entries (see IX, §2, Cor.~2 in \cite{WeilBNT}).
\end{proof}
  Note in particular that the pfaffian $\pf_\tau: \Sym(\Lambda, \tau) \to \calO$ is surjective.

% \begin{corollary}
%   Let $C$ be a commutative $\calO$-algebra which is an integral domain of characteristic $0$, then
%   the pfaffian $\pf_\tau: \Sym(\Lambda, \tau) \otimes_{\calO} C \to C$ is surjective.
% \end{corollary}
% \begin{proof}
%    Let $k$ be the quotient field of $C$. By assumption $\calO$ injects into $C$, thus $k$ is an extension field of $F$.
%    Note that the canonical map 
%   \begin{equation*}
%      \Sym(\Lambda, \tau) \otimes_\calO C \to \Sym(\Lambda, \tau) \otimes_\calO k
%   \end{equation*}
%    is injective since $\Sym(\Lambda, \tau)$ is a
%    projective $\calO$-module.
% 
%    Let $c \in C$ be given. 
%    The matrix $x = \diag(c,1,\dots,1) \in \Lambda\otimes_{\calO} C$
%    actually lies in $\Sym(\Lambda, \tau) \otimes_{\calO} C$. The lemma, applied with the field $k$,
%    yields $\pf_\tau(x) = c$.
% \end{proof}

\subsubsection{Transfer of results to $\SL_\Lambda$}
 The final step in this section is to transfer the results on non-abelian Galois cohomology with values in $\GL_\Lambda$
 to the group $G = \SL_\Lambda$.
 Our main tool is the cohomological diagram for symplectic involutions Prop.~\ref{prop:diagramSymplectic}.

\begin{lemma}\label{lem:H1nrdOnto}
  Let $v \in V \setminus \Ram_\infty(D)$ be a place of $F$. 
 Then the Pfaffian induces a bijection
    \begin{equation*}
      \pf_\tau : H^1(\tau^*, \SL_\Lambda(F_v)) \stackrel{\simeq}{\longrightarrow} \{\pm 1\}. 
   \end{equation*}
\end{lemma}
\begin{proof}
   If follows from Prop.~\ref{prop:H1GLlocalFields} that $H^1(\tau^*,\GL_\Lambda(F_v))$ is trivial.
   The cohomological diagram for symplectic involutions (see Prop.~\ref{prop:diagramSymplectic}) collapses to
   \begin{equation*}
      \begin{CD}
          \{\pm1\} @>{\delta}>> H^1(\tau^*, \SL_\Lambda(F_v) ) @>>> 1 \\
             @|                      @VV{\pf_\tau}V                 @. \\
          \{\pm1\}   @=                \{\pm1\}.     @.   \\
      \end{CD}
   \end{equation*}
   Here we used that $\nrd_\Lambda: \GL_\Lambda(F_v) \to F_v^\times$ is surjective (see (33.4) in \cite{Reiner2003}).
   By Prop.~\ref{prop:diagramSymplectic} the morphism $\delta$ is injective, and thus bijective.
\end{proof}

\begin{lemma}\label{lem:H1localSL}
   Let $v \in \Ram_\infty(D)$. The canonical map 
   \begin{equation*}
     j_*: H^1(\tau^*,\SL_\Lambda(F_v)) \to H^1(\tau^*,\GL_\Lambda(F_v))
   \end{equation*}
   is bijective.
\end{lemma}
\begin{proof}
   In this case the reduced norm takes only positive values in $F_v \cong \bbR$.
   Therefore the cohomological diagram for symplectic involutions (Prop.~\ref{prop:diagramSymplectic}) yields
   \begin{equation*}
      \begin{CD}
         1 @>>> H^1(\tau^*, \SL_\Lambda(F_v)) @>{j_*}>> H^1(\tau^*, \GL_\Lambda(F_v)) \\
         @.         @VV{\pf_\tau}V                        @VV{\pf_\tau}V              \\     
         1 @>>>   \{\pm1\}           @>{\simeq}>>        \bbR^\times/\bbR^\times_{>0} 
      \end{CD}
   \end{equation*}
   It follows directly from Corollary \ref{cor:imageJ} that $j_*$ is surjective.
   Moreover, twisting the upper row with cocycles for $H^1(\tau^*, \SL_\Lambda(F_v))$ shows that
   $j_*$ is indeed injective.
   For more details on twisting in non-abelian cohomology the reader may consult \cite[I,5.4]{Serre1964}.
   Note that twisting an involution of symplectic type gives an involution of symplectic type (see Remark \ref{rem:twistingInvolutions}).
\end{proof}

\begin{lemma}\label{lem:H1finiteLocalSL}
 Let $v$ be a finite place and let $\LP_v \subseteq \calO_v$ be the prime ideal. 
 For an integer $m \geq 0$ we define $K_v(m) := \ker(G(\calO_v) \to G(\calO_v/\LP^m_v))$.
 Then the pfaffian induces a bijection 
   \begin{equation*}
        \pf_\tau: H^1(\tau^*,K_v(m)) \isomorph \begin{cases}
                                                               \{\pm 1\} \:& \text{ if $-1 \equiv 1 \mod \LP_v^m$, }\\
                                                                 \{ 1 \}   & \text{ otherwise }.
                                                   \end{cases}
   \end{equation*}
\end{lemma}
\begin{proof}
  We start with the special case $m=0$, this is $K_v(m) = \SL_\Lambda(\calO_v)$.
  Here the claim follows just as in the proof of Lemma \ref{lem:H1nrdOnto} from
  Prop.~\ref{prop:diagramSymplectic}, Cor.~\ref{cor:H1localfiniteGL}, and the fact that 
  the reduced norm $\nrd_\Lambda: \GL_\Lambda(\calO_v) \to \calO_v^\times$ is onto (combine (14.1) and Ex.~5, p.~152 in \cite{Reiner2003}). 

  For $m \geq 1$ consider the short exact sequence of groups
  \begin{equation*}
     1 \longrightarrow K_v(m) \longrightarrow \SL_\Lambda(\calO_v) \longrightarrow \SL_\Lambda(\calO_v/ \LP_v^m) \longrightarrow 1.
  \end{equation*}
  Note that this sequence uses that the order $\Lambda$, and hence the group scheme $\SL_\Lambda$, is smooth (cf.~\ref{lem:orderIsNice}).
  We obtain a long exact sequence of pointed sets
  \begin{equation*}
      G^{\tau^*}(\calO_v) \stackrel{\pi}{\longrightarrow} G^{\tau^*}(\calO_v/ \LP_v^m) \stackrel{\delta}{\longrightarrow} H^1(\tau^*, K_v(m)) 
                  \stackrel{j_m}{\longrightarrow} H^1(\tau^*, G(\calO_v)).
  \end{equation*}
  It follows from Remark \ref{rem:symplNrdOne} that the fixed point group $G^{\tau^*}$ 
  is just the group scheme $G(\Lambda,\tau)$ defined in \ref{sec:FixedPointGroups}.
  Since the group scheme $G(\Lambda,\tau)$ is smooth (Prop.~\ref{prop:smoothFPGroups}),
  the canonical map $\pi$ is surjective, and so $\delta$ is trivial.
  Via twisting (cf.~Remark \ref{rem:twistingInvolutions}) we obtain that $j_m$ is injective.

  We use that the pfaffian is a morphism of schemes defined over $\calO$ (cf.~\ref{par:pfaffian}):
  Given a cocycle $b \in Z^1(\tau^*, K_v(m))$, we have $\pf_\tau(b) \equiv 1 \mod \LP_v^m$.
  Consequently, if $1$ and $-1$ are not congruent modulo $\LP_v^m$, then $H^1(\tau^*,K_v(m)) = \{1\}$ and the claim follows.

  Assume now that $-1 \equiv 1 \mod \LP_v^m$, then the matrix $\diag(-1,1,\dots,1)$ lies in $K_v(m)$ and has pfaffian $-1$
  (cf.~\ref{par:pfaffianDiag}).
\end{proof}

For a real place $v \in V_\infty$ we denote the associated embedding $F \to \bbR$ by $\iota_v$.
 Define 
  \begin{equation*}
     F^\times_D =\{\:x \in F^\times\:|\:\iota_v(x)>0 \text{ for all } v \in \Ram_\infty(D)\:\}.
 \end{equation*}
  By the Hasse-Schilling-Maass Theorem (cf.~(33.15) in \cite{Reiner2003}) the image of the reduced norm ${\nrd_A: A^\times \to F^\times}$ is $F^\times_D$.
  
\begin{lemma}\label{lem:H1globalSL}
  Assume that $\Ram_\infty(D)$ is not empty.
  Then the canonical morphism of pointed sets
   \begin{equation*}
       j_*: H^1(\tau^*,\SL_\Lambda(F)) \longrightarrow  H^1(\tau^*,\GL_\Lambda(F))
   \end{equation*}
   is injective. The image consists of precisely
   those classes $x \in  H^1(\tau^*,\GL_\Lambda(F))$
    which satisfy $\pf_\tau(x) = \pm 1 \cdot F^\times_D$.

  If otherwise $D$ splits at every real place, then the pfaffian
  induces a bijection
  \begin{equation*}
      \pf_\tau: H^1(\tau^*,\SL_\Lambda(F)) \stackrel{\simeq}{\longrightarrow} \{\pm 1\}.
  \end{equation*}
\end{lemma}
\begin{proof}
  Assume that $\Ram_\infty(D)$ is empty.
  By the Hasse-Schilling-Maass Theorem the reduced norm 
  $\GL_\Lambda(F) \to F^\times$ is surjective and the second assertion follows as in Lemma \ref{lem:H1nrdOnto}.

  Now we assume that $\Ram_\infty(D)$ is not empty.
  The image of the reduced norm ${\nrd_A: A^\times \to F^\times}$ is $F^\times_D$. Note that $F_D^\times$
  can not contain the element $-1$ since $\Ram_\infty(D)$ is not empty.
  Consider the cohomological diagram for symplectic involutions (Prop.~\ref{prop:diagramSymplectic})
  \begin{equation*}
    \begin{CD}
       1 @>>> H^1(\tau^*, \SL_\Lambda(F)) @>{j_*}>> H^1(\tau^*, \GL_\Lambda(F)) \\
        @.          @VV{\pf_\tau}V                     @VV{\pf_\tau}V \\
       1  @>>>      \{\pm1\}                    @>>>   F^\times/F_D^\times. 
     \end{CD}       
  \end{equation*}
    Twisting shows that the map $j_*$ is injective. The assertion about the image of $j_*$ follows
    immediately from Corollary \ref{cor:imageJ}.
\end{proof}

\begin{remark}
 Assume that $\Ram_\infty(D)$ is not empty. Let $x \in H^1(\tau^*, \GL_\Lambda(F))$ be a cohomology class.
 For every place $v \in \Ram_\infty(D)$ the class $x$ considered as a class in $H^1(\tau^*, \GL_\Lambda(F_v))$ has a local signature $(p_v, q_v)$.
 Then according to Lemma \ref{lem:H1globalSL} the class $x$ lies in
 the image of $j_*$ if and only if  
  \begin{equation*}
   q_v \equiv q_w \mod 2
  \end{equation*}
 for every pair of places $v, w \in \Ram_\infty(D)$.
 This means that either all $q_v$ are even or all $q_v$ are odd.
\end{remark}

\begin{theorem}\label{thm:H1Sequence}
   Let $K_f = \prod_{v \in V_f} K_v \subseteq G(\bbA_f)$ be the open compact subgroup associated
   with the congruence subgroup $\Gamma(\LA)$ (cf.~\ref{par:defCongruenceGroups}).
   Consider the set $\calH^1(\tau^*)$ (cf.~beginning of Section  \ref{sec:HermitianFormsAndCoho}).  
   The projection $\pi: \calH^1(\tau^*) \to H^1(\tau^*, G(F))$ is injective 
   and there is a short exact sequence of pointed sets
   \begin{equation*}
       1 \longrightarrow\: \calH^1(\tau^*)\: \stackrel{\pi}{\longrightarrow}\: H^1(\tau^*, G(F)) \:\stackrel{\pf_\tau}{\longrightarrow} \:\{\pm 1\} \:\longrightarrow 1.
   \end{equation*}
\end{theorem}
\begin{proof}
   Consider the cohomology set $H^1(\tau^*, K_\infty K_f)$, which
   agrees with the direct product $H^1(\tau^*, K_\infty) \times H^1(\tau^*, K_f)$.
   The canonical map $H^1(\tau^*, K_\infty) \to H^1(\tau^*, G_\infty)$ is bijective (see \cite{AnWang2008} or Lemma 1.4 in \cite{Rohlfs1981}).
   Moreover, for every finite place $v \in V_f$ the group $K_v$ is of the form
   \begin{equation*}
       K_v(m) = \ker\bigl(G(\calO_v) \to G(\calO_v/\LP_v^m)\bigr)
   \end{equation*}
   for some integer $m$. It follows from Lemma \ref{lem:H1finiteLocalSL} that the inclusion $K_v \to G(F_v)$ induces
   an injection $H^1(\tau^*, K_v) \to H^1(\tau^*, G(F_v))$.
   Therefore the canonical map $H^1(\tau^*,K_\infty K_f) \to H^1(\tau^*, G(\bbA))$ is injective and we conclude
   that the projection $\pi: \calH^1(\tau^*) \to  H^1(\tau^*, G(F))$ is injective.
   
   Moreover, it follows from the considerations on diagonal matrices in \ref{par:pfaffianDiag} that
   the pfaffian $\pf_\tau: H^1(\tau^*, G(F)) \to \{\pm1\}$ is surjective.
   
   It remains to understand the image of $\pi$.
   Since $\Gamma(\LA)$ is (by assumption) torsion-free, we know that $-1$ is not congruent $1$ modulo $\LA$.
   In particular there is a prime ideal $\LP$ which divides $\LA$, say $e = \nu_\LP(\LA)$, 
   such that $1$ and $-1$ are not congruent modulo $\LP^e$.
   Let $v \in V_f$ be the finite place associated with $\LP$, then $K_v = K_v(e)$ and
   $H^1(\tau^*, K_v) = \{1\}$ by Lemma \ref{lem:H1finiteLocalSL}.
   Let $\gamma \in H^1(\tau^*, G(F))$ be in the image of $\pi$, say $(x, \gamma)$ is the inverse image
   in $\calH^1(\tau^*)$. Let $x_v$ be the projection of the class $x$ to $H^1(\tau^*, K_v)$. Since
   $x$ and $\gamma$ have the same image in $H^1(\tau^*, G(\bbA))$, we can deduce that
   $\pf_\tau(\gamma) = \pf_\tau(x_v) = 1$.

   Conversely, given $\gamma \in H^1(\tau^*, G(F))$ in the kernel of the pfaffian, then $\gamma$ lies in the image of $\pi$.
   Let $c_\infty \in H^1(\tau^*,K_\infty)$ be a cohomology class such that $c_\infty$ and $\gamma$
   define the same class in $H^1(\tau^*,G_\infty)$. Let $1_f$ denote the trivial class in $H^1(\tau^*,K_f)$, then the triple $(c_\infty, 1_f, \gamma)$ is a
   class in $\calH^1(\tau^*)$ which is mapped to $\gamma$ by $\pi$.
\end{proof}

\subsection{The fixed point groups}
Up to paragraph \ref{par:IntroCharThm} the number field $F$ is assumed to be  \emph{totally real}.

\begin{definition}
  Let $R$ be a commutative $\calO$-algebra (e.g.~$\calO_v$ or $F_v$).
  For a cocycle $\gamma \in Z^1(\tau^*, G(R))$ the $R$-group scheme $G(\gamma)$ of $\tau^*|\gamma$-fixed points is defined
  by
   \begin{equation*}
       G(\gamma)(C) := \{\:g \in G(C)\:|\: g = \up{\tau^*|\gamma}g\:\}
   \end{equation*}
   for any commutative $R$-algebra $C$.
   Recall that the $\gamma$-twisted $\tau^*$-action
   is given by $\up{\tau^*|\gamma}g = \gamma \up{\tau^*} g \gamma^{-1}$.
\end{definition}
We define the symplectic group $\Sp_n$ over $\bbZ$ by
 \begin{equation*}
      \Sp_n(R) :=  \{\:g\in \GL_{2n}(R)\:|\:g^{T}Jg = J\:\},
 \end{equation*}
for every commutative ring $R$, where $J$ is the standard symplectic matrix
\begin{equation*}
    J = \begin{pmatrix}
            0_n  & 1_n\\
            -1_n & 0_n
        \end{pmatrix}.
\end{equation*}
Note that in this notation $\Sp_n$ is of rank $n$, but consists of matrices of size $2n \times 2n$.

Given a cocycle $\gamma \in Z^1(\tau^*, G(\calO))$, we want to understand the associated group scheme $G(\gamma)$.
In particular we want to calculate the Euler characteristic of congruence subgroups of this group.
We start with some basic observations and afterwards we collect all the ingredients necessary 
for an application of the adelic Euler characteristic formula (Thm.~\ref{thm:adelicFormula}).

\begin{remark}
If $\gamma \in Z^1(\tau^*, G(\calO))$, then $G(\gamma) = G(\Lambda, \tau|\gamma)$ in the notation
of Section~\ref{sec:FixedPointGroups}. The reason for this identity is that $G(\Lambda, \tau|\gamma)$ is always a closed subscheme of $\SL_\Lambda$, 
i.e.~all elements have reduced norm one (see Remark \ref{rem:symplNrdOne}).
Here $\tau|\gamma$ is the $\gamma$-twisted involution on $A$ (cf.~Remark \ref{rem:twistingInvolutions}).
Recall that $\tau|\gamma$ is of symplectic type, and that twisting and the operation $*$ commute, i.e.~${(\tau|\gamma)^* = \tau^*|\gamma}$.
\end{remark}

\begin{lemma}\label{lem:FPGroupsAreSmooth}
For every $\gamma \in Z^1(\tau^*, G(\calO))$ the group scheme $G(\gamma)$ is smooth. 
\end{lemma}
\begin{proof}
  By Lemma \ref{lem:orderIsNice} the order $\Lambda$ is $\tau$-smooth. 
  Remark \ref{rem:twistingInvolutions} implies further that $\Lambda$ is $(\tau|\gamma)$-smooth as well, and
  thus Prop.~\ref{prop:smoothFPGroups} yields that ${G(\gamma) = G(\Lambda, \tau|\gamma)}$ is smooth.
\end{proof}

\begin{lemma}\label{lem:EquivCyclesGiveIsoGroups}
 Let $R$ be a commutative $\calO$-algebra.
 Suppose the two cocycles $\gamma, \gamma'$ in $Z^1(\tau^*, G(R))$ define the same class in $H^1(\tau^*, \GL_\Lambda(R))$, then
 $G(\gamma)$ and $G(\gamma')$ are isomorphic as group schemes over $R$.
\end{lemma}
\begin{proof}
   There is $c \in \GL_\Lambda(R)$ which satisfies $\gamma' = c\gamma\up{\tau}c$.
   We define a morphism of group schemes
    $f: G(\gamma) \to G(\gamma')$ by
   \begin{equation*}
      f_C: g \mapsto c g c^{-1}
   \end{equation*}
   for every commutative $R$-algebra $C$ and all $g \in G(\gamma)(C)$.
   This map is well-defined:
   \begin{equation*}
       \up{\tau^*|\gamma'}(cgc^{-1}) = \gamma' \up{\tau^*}c \up{\tau^*}g \up{\tau^*}c^{-1} \gamma'^{-1} 
          = c \gamma \up{\tau^*}g \gamma^{-1} c^{-1} = c g c^{-1}.
   \end{equation*}
   The inverse map of $f$ is obviously given by $g \mapsto c^{-1}g c$, thus $f$ is an isomorphism.
\end{proof}

\begin{corollary}\label{cor:BasicStructureFPGroups}
   Let $\gamma \in Z^1(\tau^*, G(\calO))$ be a cocycle.
   Let $R$ be a commutative \mbox{$\calO$-algebra} with $H^1(\tau^*, \GL_\Lambda(R)) = \{1\}$,
   there is an isomorphism of $R$-group schemes
   \begin{equation*}
        G(\gamma)\times_\calO R \stackrel{\simeq}{\longrightarrow} G(1) \times_\calO R.
   \end{equation*}
   In particular, this holds if $R = \calO_v$ for $v \in V_f$ (see \ref{cor:H1localfiniteGL}).

   Moreover, if $k$ is a splitting field of $D$, then
   $G(\gamma)\times_\calO k$ is isomorphic to the symplectic group $\Sp_n\times_\bbZ k$ defined over $k$.
\end{corollary}
\begin{proof}
  The first part follows immediately from Lemma \ref{lem:EquivCyclesGiveIsoGroups}.
  For the second assertion note that we can choose a splitting $\varphi: A \otimes k \to M_{2n}(k)$ 
  such that $\varphi(\tau(x))$ equals $J \varphi(x)^T J^{-1}$, 
  where $J$ denotes the standard symplectic matrix.
\end{proof}

\subsubsection{The associated real Lie groups}
 Let $\gamma \in Z^1(\tau^*, G(\calO))$ be a cocycle. 
  Consider the real Lie group 
  \begin{equation*}
      G(\gamma)_\infty = \prod_{v \in V_\infty} G(\gamma)(F_v)
  \end{equation*}
  associated with the group $G(\gamma)$.
  
\begin{lemma}\label{lem:RealRamifiedFPGroups}
   Let $\gamma \in Z^1(\tau^*, G(\calO))$ be a cocycle and let $v \in \Ram_\infty(D)$ be a real ramified place.
   If the class of $\gamma$ in $H^1(\tau^*,G(F_v))$ has signature $(p,q)$, then
   there is an isomorphism of real Lie groups
    \begin{equation*}
        G(\gamma)(F_v) \isomorph \Sp(p,q).
    \end{equation*}
   Here $\Sp(p,q)$ is the real Lie group defined by
   \begin{equation*}
       \Sp(p,q) := \{\:g \in \GL_n(\bbH)\:|\:\conj{g}^T I_{p,q} g = I_{p,q}\:\}.
   \end{equation*}
\end{lemma}
\begin{proof}
   This follows from Lemma \ref{lem:EquivCyclesGiveIsoGroups} 
   and the description of the cohomology set $H^1(\tau^*,\GL_\Lambda(F_v))$ in Prop.~\ref{prop:H1GLlocalFields}.
\end{proof}

  For a real ramified place $v \in \Ram_\infty(D)$ let $(p_v,q_v)$ denote the local signature of
  the cohomology class of $\gamma$ in $H^1(\tau^*,G(F_v))$.
  It follows from Corollary~\ref{cor:BasicStructureFPGroups} and Lemma~\ref{lem:RealRamifiedFPGroups}
  that there is an isomorphism of real Lie groups
  \begin{equation*}
         G(\gamma)_\infty \isomorph \Sp_n(\bbR)^s \times \prod_{v \in \Ram_\infty(D)}\Sp(p_v,q_v).
  \end{equation*}
    Here $s$ denotes the number of real places of $F$ which split $D$.
    Note that $G(\gamma)_\infty$ is connected and semisimple.
  The real Lie algebra $\LG(\gamma)_\infty$ of $G(\gamma)_\infty$ is isomorphic to
  \begin{equation*}
      \LG(\gamma)_\infty \cong \Lsp(n,\bbR)^s \oplus \bigoplus_{v\in \Ram_\infty(D)} \Lsp(p_v,q_v).
  \end{equation*}
  
  Recall that every maximal compact subgroup of the real Lie group $\Sp_n(\bbR)$ 
  is isomorphic to the unitary group $\U(n)$.

  Consider the group $\Sp(n) := \Sp(n,0)$. 
  One can check that this is a compact connected semisimple real 
  Lie group (see \cite[p.~111]{Knapp2002}).
  Moreover, it is a maximal compact subgroup of the special linear group $\SL_n(\bbH)$.

  Let $p,q \geq 0$ be integers with $p+q = n$. 
  The Lie group $\Sp(p,q)$ is connected and semisimple \cite[Prop.~1.145]{Knapp2002}, and
  the compact subgroup $\Sp(p)\times\Sp(q)$ is a maximal compact subgroup.
  Given any maximal compact subgroup $K(\gamma)_\infty \subseteq G(\gamma)_\infty$, we obtain an isomorphism of
  Lie groups
  \begin{equation*}
       K(\gamma)_\infty  \isomorph \U(n)^s \times \prod_{v \in \Ram_\infty(D)} \Sp(p_v) \times \Sp(q_v).
  \end{equation*}
  
\subsubsection{The symmetric space}\label{par:symSpace}
    Consider the associated Riemannian symmetric space 
    $X(\gamma) := \bsl{K(\gamma)_\infty}{G(\gamma)_\infty}$.
    We have $\dim G(\gamma)  = n(2n+1)$ and thus 
  \begin{equation*}
             \dim G(\gamma)_\infty = n(2n+1)[F:\bbQ].
  \end{equation*}
          The dimension of the unitary group $\U(n)$ is $n^2$ and consequently
  \begin{equation*}
           \dim K(\gamma)_\infty = s n^{2}  + \sum_{v \in \Ram_\infty(D)} p_v(2p_v+1)+q_v(2q_v+1).
  \end{equation*}
     Subtraction of both dimensions yields
  \begin{equation*}
   \dim X(\gamma) = s n(n+1) + \sum_{v \in \Ram_\infty(D)} 4p_vq_v,
  \end{equation*}
     which is obviously an \emph{even} number.
 
\subsubsection{Lie algebras and complexifications}\label{par:LieAlgebrasComplex}
We complexify the Lie algebra $\LG(\gamma)_\infty$ and we obtain an isomorphism
\begin{equation*}
  \LG(\gamma)_\infty  \otimes_\bbR \bbC \cong \Lsp(n, \bbC)^{[F:\bbQ]}.
\end{equation*}
 The rank of this complex semisimple Lie algebra is ${\rk(\LG(\gamma)_{\infty,\bbC})  = n [F:\bbQ]}$.
 Let $\LK(\gamma)_\infty$ denote the Lie algebra of the maximal compact subgroup $K(\gamma)_\infty$.
 The complexification of this Lie algebra is isomorphic to
 \begin{equation*}
    \LK(\gamma)_\infty \otimes_\bbR \bbC \cong \Lgl(n,\bbC)^s \oplus \bigoplus_{v\in \Ram_\infty(D)}\Lsp(p_v,\bbC) \oplus \Lsp(q_v,\bbC).
 \end{equation*}
  The rank of $\LK(\gamma)_{\infty,\bbC}$ is $s n + \sum_{v \in \Ram_\infty(D)} p_v + q_v = n  [F:\bbQ]$.
  Thus the complexified Lie algebras $\LG(\gamma)_{\infty,\bbC}$ and $\LK(\gamma)_{\infty,\bbC}$ have equal rank.
  The Weyl groups of these complex reductive Lie algebras are well-known, in particular we get
  \begin{align*}
      |W(\LG(\gamma)_{\infty,\bbC})| &= (2^n n!)^{[F:\bbQ]}, \:\text{ and } \\
      |W(\LK(\gamma)_{\infty,\bbC})| &=  (n!)^s \prod_{v \in \Ram_\infty(D)} 2^{p_v} p_v! \cdot 2^{q_v} q_v!
  \end{align*}
  as can be found in \cite[p.~66]{Humphreys1972}. The quotient of the cardinalities of the two Weyl groups is given by
  \begin{equation*}
    \frac{ |W(\LG(\gamma)_{\infty,\bbC})|}{|W(\LK(\gamma)_{\infty,\bbC})|} = 2^{ns} \prod_{v \in \Ram_\infty(D)} \binom{n}{p_v}.
  \end{equation*}
  
\begin{remark}\label{rem:GgammaSemisimple}
 The linear algebraic $F$-group $G(\gamma)\times_\calO F$ is an inner form of the symplectic group $\Sp_n$,
 in particular it is a semisimple and simply connected group.
 Further this implies that the Tamagawa number $\tau(G(\gamma))$ is equal to one (see~\cite{Kottwitz1988}).
\end{remark}

\subsubsection{The metric form $B$}\label{par:metricFormB}
Recall that the Lie algebra of $G(\gamma)$ is a functor
$\Lie(G(\gamma))$ which assigns to a commutative $\calO$-algebra $C$ the $C$-Lie algebra
\begin{equation*}
    \Lie(G(\gamma))(C) = \{\: x \in (\Lambda \otimes_\calO C)^\times\:|\: (\tau|\gamma)(x) = -x\:\}.
\end{equation*}
For simplicity we write $\LG(\gamma)_C$ instead of  $\Lie(G(\gamma))(C)$.

Consider the non-degenerate bilinear form $B: \LG(\gamma)_F \times \LG(\gamma)_F \to F$
defined by $B(x,y) := -\frac{1}{2}\trd_A(xy)$.
Let $\iota: F \to \bbC$ be an embedding of $F$ into the field of complex numbers.
The central simple algebra $A = M_n(D)$ splits over $\bbC$ and we can choose a splitting $A \to M_{2n}(\bbC)$ 
such that $\tau|\gamma$ is the standard symplectic involution.
Via this splitting the Lie algebra $\LG(\gamma)_\bbC$ is isomorphic 
to the complex semisimple Lie algebra $\Lsp(n,\bbC)$. 

\begin{proposition}\label{prop:VolumeSpn}
  Consider the compact Lie group $\Sp(n)$ and its Lie algebra 
  \begin{equation*}
     \Lsp(n) := \{\:x \in M_n(\bbH) \:|\: \conj{x}^T + x = 0\:\}.
  \end{equation*}
  Let $B$ be the positive definite $\bbR$-bilinear form $B: \Lsp(n) \times \Lsp(n) \to \bbR$ defined
  by ${B(x,y) := -\frac{1}{2}\trd(xy)}$. With respect to the right invariant Riemann metric induced by $B$, the 
  group $\Sp(n)$ has the volume 
  \begin{equation*}
      \vol_B(\Sp(n)) = \prod_{j=1}^n \frac{(2\pi)^{2j}}{2\cdot(2j-1)!}.
  \end{equation*}
\end{proposition}
\begin{proof}
  The form $B$, extended $\bbC$-linear to $\Lsp(n,\bbC)$, is
  given by $B(x,y) = -\frac{1}{2}\Tr(xy)$.
  Recall that the Killing form $\beta$ on $\Lsp(n,\bbC)$ is the form $\beta(x,y) = (2n+2)\Tr(xy)$ (cf.~III, §8 in \cite{Helgason1978}) and
  hence $\beta = -4(n+1)B$. We conclude 
  \begin{equation*}
  \vol_\beta(\Sp(n)) = \bigl(4(n+1)\bigr)^{\frac{n(2n+1)}{2}}\vol_B(\Sp(n)). 
  \end{equation*}
  The assertion follows from Ono's formula for the volume of a compact Lie group with respect to the Killing form (see (3.4.9) in \cite{Ono1966b}), which yields
  \begin{equation*}
      \vol_\beta(\Sp(n)) = \bigl(4(n+1)\bigr)^{\frac{n(2n+1)}{2}}\prod_{j=1}^n \frac{(2\pi)^{2j}}{2\cdot(2j-1)!}.\qedhere
  \end{equation*}
\end{proof}

\subsubsection{The modulus factor}
Consider the $F$-bilinear form $B: \LG(\gamma)_F \times \LG(\gamma)_F \to F$
defined by $B(x,y) := -\frac{1}{2}\trd_A(xy)$.
In this paragraph we will calculate the global modulus factor ${\mf(B) = \prod_{v \in V_f} \mf(B)_v}$ (cf.~\ref{par:modulusFactor}).
Note that $\Lambda$ is in general not a free $\calO$-module, therefore we have to work locally.

We start with the finite places $v \in V_f$ where $D$ splits. The main observation is:
We can assume that $\Lambda \otimes_\calO \calO_v = M_{2n}(\calO_v)$ and that $\tau|\gamma$ is the standard symplectic involution.
This follows from the next Lemma.

\begin{lemma}
   Let $R$ be a complete discrete valuation ring with field of fractions $k$ of characteristic $\chr(k)\neq 2$.
   Let $\sigma$ be an involution of symplectic type on $M_{2n}(k)$ and let $\Lambda \subseteq M_{2n}(k)$ be 
   a maximal $R$-order which is $\sigma$-stable.
   
   There is an element $g \in \GL_{2n}(k)$ such that 
   \begin{enumerate}
    \item  $g \Lambda g^{-1} = M_{2n}(R)$, and 
    \item  $g \sigma(x) g^{-1} = J (gxg^{-1})^T J^{-1}$, where $J$ is the standard symplectic matrix.
   \end{enumerate}
\end{lemma}
\begin{proof}
   It follows from Theorem (17.3) in \cite{Reiner2003} that there
   is an invertible matrix ${a \in \GL_{2n}(k)}$ such that $a \Lambda a^{-1} = M_{2n}(R)$.
   Moreover, $\sigma$ is an involution of symplectic type and we can consider $\inn(a): M_{2n}(k) \to M_{2n}(k)$ 
   as a splitting of the central simple $k$-algebra $M_{2n}(k)$. There is a matrix $h \in \GL_{2n}(k)$ such that $h^T = -h$
   and ${\inn(a)(\sigma(x))= h (\inn(a)(x))^T h^{-1}}$ for every $x \in M_{2n}(k)$. 
   
   Using that $\Lambda$ is $\sigma$-stable, we see that $hM_{2n}(R)h^{-1} = M_{2n}(R)$.
   After multiplication with some power of the prime element in $R$, we can assume $h \in \GL_{2n}(R)$.
   On a free module over a complete discrete valuation ring, there is only one regular symplectic form up to isogeny ($\chr(k)\neq 2$!),
   this means that there is $b \in \GL_{2n}(R)$ such that $bhb^T = J$.
   Finally, we define $g := ba$ and observe
   \begin{equation*}
       g\sigma(x)g^{-1} = bh(axa^{-1})^Th^{-1}b^{-1} = J (b^{-1})^T (axa^{-1})^T b^T J^{-1} = J (gxg^{-1})^T J^{-1}
   \end{equation*}
    for every $x \in M_{2n}(k)$.
\end{proof}

\begin{corollary}\label{cor:SplitSp}
  Let $\gamma \in Z^1(\tau^*, G(\calO))$ be a cocycle and
  let $v \in V_f$ be a finite place of $F$ which splits $D$. 
  There is an isomorphism of group schemes over $\calO_v$:
  \begin{equation*}
      G(\gamma)\times_\calO \calO_v \:\isomorph\: \Sp_n \times_{\bbZ} \calO_v.
  \end{equation*}
\end{corollary}
\begin{proof}
   This follows directly from the previous lemma since $\tau|\gamma$ is an involution of symplectic type.
\end{proof}
 
 \begin{proposition}  \label{prop:localModulusSplit}
Let $v \in V_f$ be a finite place which splits $D$.
Consider the bilinear form $B: \LG(\gamma)_F \times \LG(\gamma)_F \to F$ defined by
 $B(x,y) = -\frac{1}{2} \trd(xy)$.
The local modulus factor (see \ref{par:modulusFactor}) is 
 \begin{equation*}
    \mf(B)_v = |2|^{-n}_v.
 \end{equation*}
\end{proposition}
\begin{proof}
  By Corollary \ref{cor:SplitSp} we can assume $G(\gamma) = \Sp_n$ over $\calO_v$.
  This means
  \begin{equation*}
   \LG(\gamma)_{\calO_v} = \Lsp(n,\calO_v) = \{\:x \in M_{2n}(\calO_v)\:|\: x^TJ + Jx = 0\:\}.
  \end{equation*}
   Note that the form $B$ is given by the analogous formula $B(x,y) = -\frac{1}{2}\Tr(xy)$.
   Recall that the elements of $\Lsp(n,\calO_v)$ are matrices of the form 
   \begin{equation*}
        \begin{pmatrix}
               a & b \\
               c & -a^T
        \end{pmatrix}
   \end{equation*}
   with $a,b,c \in M_n(\calO_v)$ where $b$ and $c$ are symmetric matrices.
   Let $E_{s,t}$ denote the elementary $2n \times 2n$ matrix with exactly one entry $1$ in position $(s,t)$.
   We choose an $\calO_v$-basis of $\Lsp(n,\calO_v)$ which is made up of the following elements:
   \begin{enumerate}[ \enumspace  (1)]
    \item   $a_{i,j} := E_{i,j} - E_{j+n,i+n}$ for all $i,j \in \{1,\dots,n\}$,
    \item   $b_{i,j} := E_{i,j+n} + E_{j, i+n}$ for all $1 \leq i<j \leq n$,
    \item   $c_{i,j} := E_{i+n,j} + E_{j+n, i}$ for all $1 \leq i<j \leq n$, and
    \item   $b_i := E_{i,i+n}$ and $c_i := E_{i+n, i}$ for all $i \in \{1,\dots,n\}$. 
   \end{enumerate}
   We evaluate the form $B$ on all the basis vectors.
   
   It is an easy observation that
   \begin{align*}
      0 &= B(a_{i,j}, c_{k,l}) = B(a_{i,j}, b_{k,l}) = B(a_{i,j}, b_{k}) = B(a_{i,j}, c_{k}) \\
       &= B(c_{i,j},c_{k,l}) = B(b_{i,j},b_{k,l}) = B(c_i,c_j) = B(b_i,b_j).
   \end{align*}
    for all $i,j,k,l$.
    Moreover, one readily verifies that $B(b_{i,j}, c_k) = B(c_{i,j},b_k) = 0$ for all $i,j,k$.
    The remaining cases yield:
    \begin{itemize}
     \item $B(a_{i,j}, a_{k,l}) = -\delta_{j,k}\delta_{i,l}$ for all $i,j,k,l \in \{1,\dots,n\}$,
     \item $B(b_{i,j},c_{k,l}) = - \delta_{i,k}\delta_{j,l}$ for all $i<j \leq n$ and $k < l \leq n$, and      
     \item $B(b_i,c_j) = -\frac{1}{2} \delta_{i,j}$ for all $i,j \in \{1,\dots,n\}$.
    \end{itemize}
     Using these results, we are able to calculate the modulus factor and obtain 
     \begin{equation*}
        \mf(B)_v = \bigl|\det\begin{pmatrix}
                                  0 & -1/2 \\
                                -1/2 & 0
                              \end{pmatrix}\bigr|_v^{n/2} = |2|^{-n}_v. \qedhere
     \end{equation*}
\end{proof}

\begin{proposition}\label{prop:localModulusRam}
  Let $v \in \Ram_f(D)$ be a finite ramified place and let $\LP \subset \calO$ be the associated prime ideal.
  The local modulus factor for the group $G(\gamma)$ and the form $B$ defined in \ref{par:metricFormB}
  is
  \begin{equation*}
        \mf(B)_v = |2|^{-n}_v \N(\LP)^{-n(n+1)/2}.
  \end{equation*}
\end{proposition}
\begin{proof}
  The $F_v$-algebra $D_v := D \otimes_F F_v$ is the unique quaternion division algebra over $F_v$ and $\Delta := \Lambda_D \otimes_\calO \calO_v$ 
  is the unique maximal order in $D_v$.

  Due to Corollary \ref{cor:BasicStructureFPGroups} we can assume that $\gamma = 1$, i.e.~$G(\gamma)\times\calO_v$ is isomorphic to
  $ H := G(1)\times\calO_v$.
  We define 
  \begin{equation*}
      \LH := \Lie(H)(\calO_v) = \{\:x \in M_n(\Delta)\:|\:\tau(x) = -x\:\}.
  \end{equation*}
  Recall that $\tau(x) = \conj{x}^T$. 
  
  Take an $\calO_v$-basis $v_0, v_1, v_2, v_3$ of $\Delta$ such that $\trd_D(v_0) = 1$ and $\trd_D(v_i)=0$ for $i=1,2,3$.
  Such a basis exists since $\trd_D: \Delta \to \calO_v$ is surjective (maximal orders are smooth, see Prop.~\ref{prop:maxImpliesSmooth}).
  We construct an $\calO_v$-basis of the Lie algebra $\LH$, which consists of the following elements
  \begin{enumerate}[\enumspace (1)]
   \item $a_{s,i} := v_s E_{i,i}$ for all $s\in\{1,2,3\}$ and $i \in \{1,\dots,n\}$, and
   \item $b_{s,i,j} := v_s E_{i,j} - \conj{v_s} E_{j,i}$ for all $s\in\{0,1,2,3\}$ and $i,j \in \{1,\dots,n\}$ with $i < j$.
  \end{enumerate}
   We calculate the form $B$ on all basis vectors.
   Observe that $B(a_{s,i}, b_{t,k,l}) = 0$ for all~$s,t,i,k,l$.
   Moreover, for $s,t \in \{1,2,3\}$ and $i,j \in \{1,\dots,n\}$ we find
   \begin{equation*}
       B(a_{s,i},a_{t,j}) = -\frac{1}{2}\trd\bigl( v_s E_{i,i} v_t E_{j,j} \bigr) = -\frac{1}{2}\delta_{i,j} \trd_D(v_s v_t).
   \end{equation*}
   Finally, let $s,t \in \{0,1,2,3\}$ and let $i,j,k,l \in \{1,\dots,n\}$ with $i<j$ and $k<l$. 
   We obtain
   \begin{equation*}
       B(b_{s,i,j},b_{t,k,l}) =  \frac{1}{2} \trd_D(\conj{v_s}v_t+v_s\conj{v_t})\delta_{i,k}\delta_{j,l} = \trd_D(v_s\conj{v_t})\delta_{i,k}\delta_{j,l}.
   \end{equation*}
   Summing up we obtain a formula for the modulus factor
   \begin{equation}\label{eq:modulusRamified}
       \mf(B)^2_v = \bigl|\frac{1}{8}\det(\trd(v_sv_t))_{s,t=1,2,3}\bigr|_v^n \cdot \bigl|\det(\trd(v_s\conj{v_t}))_{s,t=0,1,2,3}\bigr|^{n(n-1)/2}_v
   \end{equation}
   Since the elements $\conj{v_0}, \conj{v_1},\conj{v_2},\conj{v_3}$ form an $\calO_v$-basis of $\Delta$ as well, 
   we see that the second term $\bigl|\det(\trd(v_s\conj{v_t}))_{s,t=0,1,2,3}\bigr|_v$ is the valuation of the discriminant of $\Delta$.
   It is known that the discriminant of $\Delta$ is $\LP_v^2$ (see (14.9) in \cite{Reiner2003}). 
  
   To calculate the first term in equation \eqref{eq:modulusRamified} we consider $w_0 := 1$ and we define $w_s= v_s$ for $s=1,2,3$.
   Note that $w_0, w_1,w_2,w_3$ is in general not an $\calO_v$-basis of $\Delta$ since $\trd_D(1) = 2$ need not be a unit in $\calO_v$.
   We can write
   \begin{equation*} 
       w_0 = 1 = r_0v_0 + r_1 v_1 + r_2v_2+ r_3v_3
   \end{equation*}
   for certain $r_0,r_1,r_2,r_3$ in $\calO_v$. Applying the reduced trace we get ${2 = \trd_D(1) = r_0}$.
   Furthermore, this implies that the matrix $\bigl(\trd(w_sw_t)\bigr)_{s,t=0,1,2,3}$ can be written as a product of matrices
   \begin{equation*}
       \begin{pmatrix}
           2 & r_1 & r_2 & r_3 \\
           0 &  1  & 0   & 0 \\
           0 &  0  & 1   & 0 \\
           0 &  0  & 0   & 1   
       \end{pmatrix}
         \bigl(\trd(v_iv_j)\bigr)_{i,j=0,1,2,3}
       \begin{pmatrix}
           2 & 0 & 0 & 0 \\
           r_1 &  1  & 0   & 0 \\
           r_2 &  0  & 1   & 0 \\
           r_3 &  0  & 0   & 1   
       \end{pmatrix}.
   \end{equation*}
   Note that 
   \begin{equation*}
       \bigl(\trd(w_sw_t)\bigr)_{s,t=0,1,2,3} = \begin{pmatrix}
                                                    2 & 0 & 0 & 0\\
                                                    0 & \trd(v_1v_1) & \trd(v_1v_2) & \trd(v_1v_3) \\
                                                    0 & \trd(v_2v_1) &\trd(v_2v_2) &\trd(v_2v_3)  \\
                                                    0 & \trd(v_3v_1) &\trd(v_3v_2) &\trd(v_3v_3)  
                                                  \end{pmatrix}.
   \end{equation*}
   We deduce that $\bigl|\det(\trd(v_sv_t))_{s,t=1,2,3}\bigr|_v= |2|_v\N(\LP)^{-2}$.
   In total the local modulus factor is
   \begin{equation*}
        \mf(B)_v = |2|^{-n}_v \N(\LP)^{-n - n(n-1)/2} = |2|^{-n}_v \N(\LP)^{- n(n+1)/2}. \qedhere
   \end{equation*}
\end{proof}

\begin{corollary}\label{cor:globalModulus}
   Let $\gamma \in Z^1(\tau^*,G(\calO))$ be a cocycle.
   The global modulus factor $\mf(B)$ for the group $G(\gamma)$
   with respect to the form $B$ defined in \ref{par:metricFormB} is
   \begin{equation*}
       \mf(B)  = 2^{n[F:\bbQ]} (-1)^{rn(n+1)/2}\Drd{D}^{-n(n+1)/2},
   \end{equation*}
   where $\Drd{D}$ denotes the signed reduced discriminant of $D$ (see \ref{def:redDiscr}).
\end{corollary}
\begin{proof}
  By Prop.~\ref{prop:localModulusSplit}, Prop.~\ref{prop:localModulusRam},
  and an application of the product formula
  we obtain
  \begin{equation*}
     \mf(B) = \prod_{v \in V_f} |2|_v^{-n} \prod_{\LP  \in \Ram_f(D)}\N(\LP)^{-n(n+1)/2} = 2^{n[F:\bbQ]} \prod_{\LP  \in \Ram_f(D)}\N(\LP)^{-n(n+1)/2}. \qedhere
  \end{equation*}
\end{proof}

\subsubsection{The Euler characteristic of the fixed point groups}\label{par:IntroCharThm}
Let $\gamma \in Z^1(\tau^*,G(\calO))$ be a cocycle. 
We are now able to compute the Euler characteristic of torsion-free arithmetic subgroups of $G(\gamma)$.
In the next theorem we give a precise formula for principal congruence subgroups. 
More general subgroups can be treated analogously.

For the next theorem the number field $F$ need not be totally real.
Let $\LA \subset \calO$ be a proper ideal.
For a finite place  $v \in V_f$ we define $K_v(\gamma, \LA)$ to be the kernel of the reduction $G(\gamma)(\calO_v) \to G(\gamma)(\calO_v/\LA\calO_v)$.
Note that $K_v(\gamma, \LA) = G(\gamma)(\calO_v)$ for almost all places.
The group 
\begin{equation*}
  K_f(\gamma,\LA) :=  \prod_{v\in V_f} K_v(\gamma,\LA)
\end{equation*}
is an open compact subgroup of the locally compact group $G(\gamma)(\bbA_f)$.
This subgroup is given by a local datum $(U,\alpha)$ (cf.~\ref{par:congruenceGroups}).
Let $v \in V_f$ be a finite place and let $\LP$ be the associated prime ideal.
Let $e = \nu_\LP(\LA)$ be the exponent of $\LP$ in $\LA$.
We have $\alpha_v = 1$ and $U_v = G(\gamma)(\calO/\LP)$ if $e = 0$, otherwise $\alpha_v = e$ and
$U_v = \{1\} \subseteq  G(\gamma)(\calO/\LP^e)$.

Let $G(\gamma)_\infty = \prod_{v \in V_\infty} G(\gamma)(F_v)$ and let $K(\gamma)_\infty \subseteq G(\gamma)_\infty$
be a maximal compact subgroup. For every real ramified place $v \in \Ram_\infty(D)$ we denote the local signature of the
 class of $\gamma$ in $H^1(\tau^*,G(F_v))$ by $(p_v,q_v)$ (cf.~\ref{def:Signature}).
 
 \begin{theorem}\label{thm:EulerCharFPGroups}
   Assume that $G(\gamma)(F)$ acts freely on 
   $\bsl{K(\gamma)_\infty K_f(\gamma, \LA)}{G(\gamma)(\bbA)}$.
   The Euler characteristic of the double quotient space
   \begin{equation*}
        S(\LA) := \bsl{K(\gamma)_\infty K_f(\gamma, \LA)}{G(\gamma)(\bbA)}/G(\gamma)(F)
   \end{equation*}
    is non-zero if and only if $F$ is totally real.
    In this case the following formula holds
   \begin{equation*}
        \chi\bigl(S(\LA)\bigr) 
           = 2^{-nr} \N(\LA)^{n(2n+1)}\Drd{D}^{n(n+1)/2}\prod_{v\in\Ram_\infty(D)} \binom{n}{p_v} 
           \prod_{j=1}^n M(j,\LA,D),      
   \end{equation*}
    where $M(j,\LA,D)$ is defined as
   \begin{equation*}
      M(j,\LA,D) := \zeta_F(1-2j) \prod_{\LP | \LA}\bigl(1-\frac{1}{\N(\LP)^{2j}}\bigr)
      \prod_{\substack{\LP \in \Ram_f(D) \\ \LP \nmid \LA}} \bigl(1+(\frac{-1}{\N(\LP)})^{j}\bigr).
   \end{equation*}
   Here $r$ is the number of real places of $F$ where $D$ is ramified.
   The sign of $\chi(S(\LA))$ is $(-1)^{sn(n+1)/2}$, where $s$ denotes the number of real places where $D$ splits.
\end{theorem}
\begin{proof}
    It follows from Remark \ref{rem:NoComplexPlaces} that the Euler characteristic vanishes whenever $F$ has a complex place. 
    Therefore we may assume that $F$ is totally real. We want to apply the adelic Euler characteristic formula (Thm.~\ref{thm:adelicFormula}).
    We know that $G(\gamma)$ is a smooth group scheme over $\calO$ (see Lem.~\ref{lem:FPGroupsAreSmooth}).
    Further $G \times_\calO F$ is an inner form of the symplectic group, and is thus a semisimple and simply connected 
    algebraic group of dimension $d=n(2n+1)$ (cf.~Remark \ref{rem:GgammaSemisimple}).
    Note further, that by assumption $G(\gamma)(F)$ acts freely on 
    $\bsl{K(\gamma)_\infty K_f(\gamma, \LA)}{G(\gamma)(\bbA)}$.
    
    Moreover, we observe that $\dim X(\gamma)$ is even (cf.~\ref{par:symSpace})
    and that the complexified Lie algebras $\LK(\gamma)_\infty\otimes\bbC$ and $\LG(\gamma)_{\infty,\bbC}$
    have equal rank (cf.~\ref{par:LieAlgebrasComplex}).
    We conclude that the Euler characteristic does not vanish and we can apply Theorem \ref{thm:adelicFormula}. 
   
    We fix the non-degenerate bilinear form $B: \LG(\gamma)_F \times \LG(\gamma)_F  \to F$
    defined by $B(x,y) := -\frac{1}{2}\trd_A(xy)$.
    It is easy to see that the compact dual group $G(\gamma)_u$ of $G(\gamma)_\infty$ is 
    isomorphic to $\Sp(n)^{[F:\bbQ]}$.
    Further note that $B$ is given by the same formula on each factor of the compact dual group.
    Therefore the volume is
    \begin{equation*}
          \vol_B(G(\gamma)_u) = \Bigl(\prod_{j=1}^n \frac{(2\pi)^{2j}}{2 \cdot (2j-1)!}\Bigr)^{[F:\bbQ]}
    \end{equation*}
    according to Prop.~\ref{prop:VolumeSpn}.
    Using the global modulus factor, calculated in Cor.~\ref{cor:globalModulus}, 
    and the quotient of the orders of the involved Weyl groups, derived in \ref{par:LieAlgebrasComplex}, the
    adelic formula yields
    \begin{equation}\label{eq:tmpFormula}
       \begin{aligned}
        \chi(S(\LA)) =& (-1)^{[F:\bbQ]n(n+1)/2} \disc{F}^{d/2} 2^{ns} \prod_{v \in \Ram_\infty(D)}\binom{n}{p_v} \\
                      &\cdot \Bigl(\prod_{j=1}^n \frac{2 \cdot (2j-1)!}{(2\pi)^{2j}}\Bigr)^{[F:\bbQ]} 2^{-n[F:\bbQ]} \Drd{D}^{n(n+1)/2}
                      \prod_{\LP\in V_f} \frac{\N(\LP)^{d\alpha_{\LP}}}{|U_\LP|}.
       \end{aligned}
    \end{equation}
    Here $s$ denotes the number of real places of $F$ which split $D$. The only terms that can be negative are $(-1)^{[F:\bbQ]n(n+1)/2}$
    and the signed reduced discriminant. Consequently, the sign of the Euler characteristic is $(-1)^{sn(n+1)/2}$.

    Let $v\in V_f$ be a finite place with associated prime ideal $\LP$ and
    consider $\frac{\N(\LP)^{d\alpha_{\LP}}}{|U_\LP|}$.
    
    Case (a):
      $D$ splits at $v$ and $\LP$ does not divide $\LA$.
        In this case $\alpha_\LP = 1$ and $U_\LP = G(\gamma)(\calO/\LP)$.
        Since $G(\gamma)$ is isomorphic to $\Sp_n$ over $\calO_v$ (see Cor.~\ref{cor:SplitSp}),
        there is an isomorphism of finite groups
        $G(\gamma)(\calO/\LP) \cong \Sp_n(\calO/\LP)$.
        From 3.5 in \cite{Wilson2009} we deduce that 
       \begin{equation*}
            \bigl|G(\gamma)(\calO/\LP)\bigr| = \N(\LP)^d \prod_{j=1}^n\bigl(1-\frac{1}{\N(\LP)^{2j}}\bigr).
       \end{equation*}
 
     Case (b): $D$ is ramified at $v$ and $\LP$ does not divide $\LA$.
         In this situation we have $\alpha_\LP = 1$ and $U_\LP = G(\gamma)(\calO/\LP)$.
         Let $k= \calO/\LP$ be the finite residue class field and let $\ell / k$ be the unique quadratic extension.
         It is an easy exercise to show that $G(\gamma)(\calO/\LP)$ isomorphic to a semidirect product
         $U(\ell/k) \ltimes \Sym_n(\ell)$, where $U(\ell/k)$ denotes the unitary group of the quadratic extension $\ell/k$ and
         $\Sym_n(\ell)$ denotes the abelian group of symmetric $(n\times n)$-matrices with entries in $\ell$.
         Therefore (using 3.6 in \cite{Wilson2009}) we get
          \begin{equation*}
            \bigl|G(\gamma)(\calO/\LP)\bigr| = \N(\LP)^d \prod_{j=1}^n\bigl(1-\frac{(-1)^j}{\N(\LP)^{j}}\bigr).
          \end{equation*}

     Case (c): $\LP$ divides $\LA$.
         In this case $\alpha_v = \nu_\LP(\LA)$ and $|U_\LP|  = 1$.
         Consequently,
        \begin{equation*}
            \frac{\N(\LP)^{d\alpha_{\LP}}}{|U_\LP|} = N(\LP)^{d\nu_\LP(\LA)}.
        \end{equation*}

     The product of these terms is
     \begin{equation*}
         \prod_{\LP\in V_f} \frac{\N(\LP)^{d\alpha_{\LP}}}{|U_\LP|} = \N(\LA)^d \prod_{j=1}^n\Bigl(\zeta_F(2j) 
          \prod_{\LP | \LA}\bigl(1-\frac{1}{\N(\LP)^{2j}}\bigr)
      \prod_{\substack{\LP \in \Ram_f(D) \\ \LP \nmid \LA}} \bigl(1+(\frac{-1}{\N(\LP)})^{j}\bigr) \Bigr).
     \end{equation*}
    Here $\zeta_F$ denotes the zeta function of the number field $F$.

    Note that $d = n(2n+1) = \sum_{j=1}^n 4j-1$ and so $\disc{F}^{d/2} = \prod_{j=1}^n\disc{F}^{(4j-1)/2}$.
    The functional equation of the zeta function of the totally real number field $F$ (see VII.§6, Thm.~3 in \cite{WeilBNT}) yields 
    \begin{equation*}
        \zeta_F(2j)\disc{F}^{(4j-1)/2}\Bigl(\frac{2 \cdot (2j-1)!}{(2\pi)^{2j}}\Bigr)^{[F:\bbQ]} = (-1)^{j[F:\bbQ]}\zeta_F(1-2j)
    \end{equation*}
    for every integer $j \geq 1$.
    Using this we see that
    \begin{equation*}
        \disc{F}^{d/2}\Bigl(\prod_{j=1}^n \frac{2 \cdot (2j-1)!}{(2\pi)^{2j}}\Bigr)^{[F:\bbQ]} \prod_{j=1}^n\zeta_F(2j) = 
         (-1)^{[F:\bbQ]n(n+1)/2}\prod_{j=1}^n \zeta_F(1-2j). 
    \end{equation*}
    Substitute this into equation \eqref{eq:tmpFormula}, then a simple calculation proves the claim.  
\end{proof}

\subsection{Proof of the main Theorem} 
  The notation and assumptions are those of the introduction.
  As usual $F$ denotes an algebraic number field and $\calO$ denotes its ring of integers.
  Let $D$ be a quaternion algebra defined over $F$ and let $\Lambda_D \subseteq D$ be a maximal $\calO$-order.
  Let $n \geq 1$ be an integer, we consider the central simple $F$-algebra $A = M_n(D)$
  and the maximal $\calO$-order $\Lambda = M_n(\Lambda_D)$.
  Further $G:= \SL_\Lambda$ is the smooth $\calO$-group scheme defined as the kernel of the reduced norm over the order $\Lambda$ (cf.~\ref{def:SL}).

  We say that the quaternion algebra $D$ over $F$ is \emph{totally definite},
  if $F$ is totally real and $D$ ramifies at every real place of $F$.

 The algebraic group $G \times_\calO F$ has strong approximation since it is an \mbox{$F$-simple}, simply connected group
 and $G_\infty \cong \SL_{2n}(\bbR)^s \times \SL_n(\bbH)^r \times \SL_{2n}(\bbC)^t$ is not compact.
 Since the group $\SL_1(\bbH)$ is compact, we need the assumption that $n \geq 2$ if $D$ is totally definite. 

 Let $K_\infty \subseteq G_\infty$ be a $\tau^*$-stable maximal compact subgroup.
 Further, let $K_f$ be the open compact subgroup of $G(\bbA_f)$,
 which satisfies ${\Gamma(\LA) = K_f \cap G(F)}$ (see \ref{par:defCongruenceGroups}).
 Since $\Gamma(\LA)$ is torsion-free and $\tau^*$-stable, we can apply Theorem \ref{thmLefschetznumberRationalRepAdelic} and we obtain
 \begin{equation}\label{ch5:eq:FPFormula}
     \calL(\tau^*,\Gamma(\LA),W) \:=\: \sum_{\eta \in \calH^1(\tau^*)} \chi\bigl(\vartheta^{-1}(\eta)\bigr) \Tr(\tau^*|W(\gamma_\eta)).
 \end{equation}
  Here $\gamma_\eta$ is any representative of the $H^1(\tau^*,G(F))$ component of $\eta$
  and 
  \begin{equation*}
       \vartheta: \bigl(\bsl{K_\infty K_f}{G(\bbA)}/G(F) \bigr)^{\tau^*} \to \calH^1(\tau^*)
  \end{equation*}
  is the surjective continuous map defined in Section \ref{sec:RohlfsMethod}.

  By Theorem \ref{thm:H1Sequence} the projection $\pi: \calH^1(\tau^*) \to H^1(\tau^*,G(F))$ 
  is injective and there is an exact sequence of pointed sets
  \begin{equation}\label{eq:exactH1seq}
       1 \longrightarrow\: \calH^1(\tau^*)\: \stackrel{\pi}{\longrightarrow}\: H^1(\tau^*, G(F)) \:\stackrel{\pf_\tau}{\longrightarrow} \:\{\pm 1\} \:\longrightarrow 1.
   \end{equation}
   We deduce that, given a class $\eta\in\calH^1(\tau^*)$, every representative $\gamma_\eta \in \pi(\eta)$ has pfaffian one, and hence 
   they all describe the trivial class in $H^1(\tau^*, G(\alg{F}))$. Thus there is some $g \in G(\alg{F})$ such that $\gamma_\eta = g^{-1}\up{\tau^*}g$.
   It follows that ${\Tr(\tau^*|W(\gamma_\eta)) = \Tr(\tau^*|W)}$ 
   since $\tau^*|\gamma_\eta = \rho(g)^{-1}\circ\tau^*\circ\rho(g)$ on $W$.
   
   As a next step we describe the fixed point components. Let $\eta \in \calH^1(\tau^*)$, 
   via strong approximation we can choose representing cocycles $k_\eta$ in $Z^1(\tau^*,K_\infty K_f)$ and $\gamma_\eta$ in $Z^1(\tau^*, \Gamma(\LA))$,
   and an element $a_\infty \in G_\infty$ such that
   \begin{equation*}
    \eta = ([k_\eta],[\gamma_\eta])  \quad\text{  and  }\quad  \up{\tau^*}a_\infty = k_\eta^{-1} a_\infty \gamma_\eta.
     \end{equation*}
     We write $k_\eta = k_\infty k_0$ with $k_\infty \in K_\infty$ and $k_0 \in K_f$. Note that $k_0 = \gamma_\eta$
    considered as elements in $G(\bbA_f)$.
    By Lemma \ref{lem:StructureFPComponents} there is a homeomorphism
    \begin{equation*}
        \vartheta^{-1}(\eta) \isomorph \bsl{(a_\infty^{-1}K_\infty^{\tau^*|k_\infty}a_\infty) K_f(\gamma_\eta, \LA)}{G(\gamma_\eta)(\bbA)}/G(\gamma_\eta)(F).
    \end{equation*}
    In fact $(a_\infty^{-1}K_\infty^{\tau^*|k_\infty}a_\infty)$ is a maximal compact subgroup
    of $G(\gamma_\eta)_\infty$.
    
    Let $v\in \Ram_\infty(D)$ and let $(p_v,q_v)$ denote the local signature of $\gamma_\eta$ at $v$.
    By Theorem~\ref{thm:EulerCharFPGroups} the Euler characteristic of the fixed point component is zero 
    if $F$ has a complex place. If $F$ is totally real, which we assume from now on, then 
    \begin{equation*}
        \chi(\vartheta^{-1}(\eta)) = 2^{-nr} \N(\LA)^{n(2n+1)}\Drd{D}^{n(n+1)/2}\prod_{v\in\Ram_\infty(D)} \binom{n}{p_v} 
           \prod_{j=1}^n M(j,\LA,D).
    \end{equation*}
    
    The short exact sequence \eqref{eq:exactH1seq}, in combination with the Hasse principle (Prop.~\ref{prop:HassePrinciple})
    and Lemma \ref{lem:H1globalSL}, shows that the map which takes all the local signatures
    at the real ramified places induces a bijection
   \begin{equation*}
       \calH^1(\tau^*) \isomorph \prod_{v \in \Ram_\infty(D)} \{\:(p_v,q_v)\:|\: p_v + q_v = n \:\text{ and $q_v$ even}\:\}.
   \end{equation*}
   The following identity can be easily verified
   \begin{equation*}
       \sum_{\eta \in \calH^1(\tau^*)} \prod_{v \in \Ram_\infty(D)} \binom{n}{q_v} \:=\: \sum_{u_1, \dots, u_r=0}^{[\frac{n}{2}]} \prod_{i=1}^r \binom{n}{2u_i} = 2^{r(n-1)}.
   \end{equation*}
    As a final step we substitute all results in formula \eqref{ch5:eq:FPFormula} and we observe:
   \begin{equation*}
       \calL(\tau^*,\Gamma(\LA),W) \: = \:
         2^{-r} \N(\LA)^{n(2n+1)} \Drd{D}^{n(n+1)/2} \Tr(\tau^*|W) \prod_{j=1}^n M(j,\LA, D).
   \end{equation*}
   Note that the Lefschetz number is non-zero precisely when $F$ is totally real and $\Tr(\tau^*|W)$ does not vanish.

\subsection{The growth of the total Betti number}\label{sec:growthOfBettinumber}
There are many recent results on the asymptotic behaviour of Betti numbers of arithmetic groups.
Most of these results are upper bound results -- a strong asymptotic upper bound was obtained by Calegari-Emerton \cite{CalegariEmerton2009}.
However, there are no strong lower bound results. It seems that the only available lower bound results are non-vanishing results for certain degrees in the cohomology.
Indeed, there is a geometric method to construct cohomology classes in a given degree for cocompact arithmetic groups.
This method originated from the work of Millson and Raghunathan \cite{MillsonRaghunathan1981}
and has been further elaborated by Rohlfs and Schwermer \cite{RohlfsSchwermer1993}.
Another result that can be interpreted as a result on lower bounds has been obtained by Venkataramana \cite{Venky2008}.
In this last section we prove Corollary \ref{cor:GrowthBettiNumber} to show that Lefschetz numbers provide asymptotic lower bounds
for the total Betti number.
The only remaining step is to relate the Lefschetz number to the index of the congruence subgroup $\Gamma(\LA)$.
Let $F$ be a totally real number field. If $D$ is totally definite we assume $n \geq 2$ such that $G = \SL_\Lambda$ has strong approximation.

\begin{lemma}\label{lem:IndexGamma}
   The index $[G(\calO):\Gamma(\LA)]$ of $\Gamma(\LA)$ in $G(\calO)$ is 
   \begin{equation*}
    \N(\LA)^{4n^2-1} 
    \prod_{\substack{\LP | \LA \\ \LP \notin \Ram_f(D)}}\Bigl( \prod_{j=2}^{2n} (1-\frac{1}{\N(\LP)^{j}}) \Bigr) 
    \prod_{\substack{\LP | \LA \\ \LP \in \Ram_f(D)}}\Bigl( (1+\frac{1}{\N(\LP)}) \prod_{j=2}^n (1-\frac{1}{\N(\LP)^{2j}}) \Bigr).
   \end{equation*}
     In particular, the term $[G(\calO):\Gamma(\LA)]N(\LA)^{-4n^2+1}$ is bounded from above and from below independent of $\LA$,
   \begin{equation*}
             \prod_{j=2}^{2n} \zeta_F(j)^{-1} \leq [G(\calO):\Gamma(\LA)]N(\LA)^{-4n^2+1}  \leq  \prod_{\LP \in \Ram_f(D)}(1+\frac{1}{\N(\LP)}).
   \end{equation*}                                               
\end{lemma}
 \begin{proof}
    Using the smoothness of the group scheme combined with strong approximation, there is a short exact sequence of groups
    \begin{equation*}
        1 \longrightarrow \Gamma(\LA) \longrightarrow G(\calO) \longrightarrow G(\calO/\LA) \longrightarrow 1, 
    \end{equation*}
    from which we deduce $[G(\calO):\Gamma(\LA)] = \prod_{\LP | \LA} |G(\calO_\LP / \LA \calO_\LP)|$.
     Let $\LP$ be a prime ideal which divides $\LA$, say $\nu_\LP(\LA) = e \geq 1$.
     Then ${\calO_\LP/\LA\calO_\LP \cong \calO_\LP/\LP^e\calO_\LP}$
     and it follows from the smoothness of $G$ that
     \begin{equation*}
          |G(\calO_\LP/\LP^e\calO_\LP)| = \N(\LP)^{(e-1)d} |G(\calO_\LP/\LP\calO_\LP)|
     \end{equation*}
     where $d$ is the dimension of the group $G\times_\calO F$ 
     (use \cite[2.1]{Oesterle1984}).
     The dimension of $G$ is $d = 4n^2 -1$.

    If $\LP \in \Ram_f(D)$, then one can show that
     \begin{equation*}
             |G(\calO_\LP/\LP\calO_\LP)| = \N(\LP)^{4n^2-1}(1+\N(\LP)^{-1}) \prod_{j=2}^n(1-\N(\LP)^{-2j}).
     \end{equation*}
   
    If otherwise $\LP \notin \Ram_f(D)$, then $G \times_\calO \calO_\LP$ is isomorphic
    to the special linear group~$\SL_{2n}$.
    We deduce that
  \begin{equation*}
          |G(\calO_\LP/\LP\calO_\LP)| = \N(\LP)^{4n^2-1} \prod_{j=2}^{2n}(1-\N(\LP)^{-j})
  \end{equation*}
     due to 3.3.1 in \cite{Wilson2009}.
     Now the assertions can be readily verified.
 \end{proof}

 \begin{proof}[Proof of Corollary \ref{cor:GrowthBettiNumber}]
   Since $\Gamma_0(\LA)$ is a subgroup
   of finite index in $\Gamma(\LA)$, we obtain (from VII, Prop.~6 in \cite{Serre1979}) that $b_i(\Gamma(\LA)) \leq b_i(\Gamma_0(\LA))$.
    It follows directly from the main theorem, that
    there is a positive real number $b > 0$, depending on $F$, $D$ and $n$, such that 
   \begin{equation*}
           b \N(\LA)^{n(2n+1)} \leq |\calL(\tau^*,\Gamma(\LA),\bbC)|.
       \end{equation*}
   for every ideal $\LA \subseteq \calO$ such that $\Gamma(\LA)$ is torsion-free.
   Since $B(\Gamma(\LA)) \geq |\calL(\tau^*,\Gamma(\LA),\bbC)|$, it follows from Lemma \ref{lem:IndexGamma} that
   \begin{equation*}
        B(\Gamma(\LA)) \geq a [G(\calO):\Gamma(\LA)]^{\frac{n(2n+1)}{4n^2-1}}
   \end{equation*}
   for some $a>0$ depending on $F$, $D$ and $n$.
   We obtain
   \begin{align*}
        B(\Gamma_0(\LA)) 
         &\geq a [G(\calO):\Gamma(\LA)]^{\frac{n(2n+1)}{4n^2-1}} \geq a [G(\calO)\cap\Gamma_0:\Gamma_0(\LA)]^{\frac{n(2n+1)}{4n^2-1}}\\
         &= a  \bigl([\Gamma_0:G(\calO)\cap\Gamma_0]^{-1}  [\Gamma_0:\Gamma_0(\LA)]\bigr)^{\frac{n(2n+1)}{4n^2-1}}.
   \end{align*}
   We define $\kappa = a  [\Gamma_0:G(\calO)\cap\Gamma_0]^{-\frac{n(2n+1)}{4n^2-1}}$.
\end{proof}

\subsection*{Acknowledgements}
I would like to thank Professor J. Schwermer for his support during my thesis work, upon which this article is based.

\newpage
\providecommand{\bysame}{\leavevmode\hbox to3em{\hrulefill}\thinspace}
\providecommand{\MR}{\relax\ifhmode\unskip\space\fi MR }
% \MRhref is called by the amsart/book/proc definition of \MR.
\providecommand{\MRhref}[2]{%
  \href{http://www.ams.org/mathscinet-getitem?mr=#1}{#2}
}
\providecommand{\href}[2]{#2}

\end{document}